\documentclass[microtype]{gtpart}     % Basic GT/GTM/AGT style
\usepackage{graphicx}  %%% the recommended graphics package
\usepackage{graphicx,amssymb,upref}
\usepackage{here}
\allowdisplaybreaks

\title{Proof of the volume conjecture for twist knots}

\author{Sukuse Abe}
\givenname{}
\surname{}
\address{SHINSHU UNIVERSITY
3-1-1 Asahi, Matsumoto, Nagano
390-8621 JAPAN}
\email{k9250979@kadai.jp}
\urladdr{}


\arxivreference{}

\newtheorem{thm}{Theorem}[section]    % Standard theorem environment
\newtheorem{lem}[thm]{Lemma}          % Lemma environment with numbering 
\newtheorem{conj}[thm]{Conjecture}
\newtheorem{cor}[thm]{Corollary}
\theoremstyle{definition}
\newtheorem{defn}[thm]{Definition}    % Definition environment with 
\begin{document}

\begin{abstract}    % type your abstract below
We prove the volume conjecture for every twist knots by using an equivalence relation, complex analysis, analytic continuation, and the method of steepest descent on the basis of colored Jones polynomials.
\end{abstract}

\maketitle

%%%%%%%%%%%%%%%%%%%%   Start of main body of article

\section{Introduction}
H.Murakami-J.Murakami \cite{MM} reformulated Kashaev's conjecture as the volume conjecture.
In this paper, we study the volume conjecture of every oriented twist knots.
The volume conjectures has so far proven to be true only for the following four sorts of knot:
the figure-eight knot by Ekholm (see \cite{15}),
$5_{2}$ by Kashaev-Yokota \cite{14} and Ohtsuki \cite{19}, torus knots by Kashaev-Tirkkonen\cite{13}
and the whitehead doubles of non-trivial torus knots by Zheng \cite{16}.
These volume conjectures were solved by calculating colored Jones polynomial.
We solve the volume conjecture for every twist knots by using some equivalence relations, complex analysis, analytic continuations, and the method of steepest descent on the basis of colored Jones polynomials.
The following is  the volume conjecture:
\begin{conj}
Let $K$ be an oriented hyperbolic knot, $J_N(K,q)$ be the colored Jones polynomial of $K$ and 
${\rm Vol}(K)$ be the hyperbolic volume of $S^3 \setminus K$. Then we obtain the following equality holds:
\begin{equation}\label{50}
\lim_{N\to \infty} \frac{2 \pi \log |J_{N}(K,{\rm exp}(2 \pi \sqrt{-1}/N))|}{N}={\rm Vol}(K).
\end{equation}
\end{conj}
The purpose of this paper is to prove that the volume conjecture is true for every oriented twist knots.
Namely, we prove the following theorem.
\begin{thm}\label{97}
Let $K$ be a twist knot, $J_N(K,q)$ be the colored Jones polynomial of $K$, ${\rm Vol}(K)$ be the hyperbolic volume of $S^3 \setminus K$. Then the above equality (\ref{50}) holds.
\end{thm}

Let $K$ be a twist knot as in Figure \ref{1006}. 
Note that positive numbers correspond to left-handed twists and negative numbers correspond to right-handed twists respectively.
\begin{figure}[H]
 \begin{center}
  \includegraphics[clip,width=130pt]{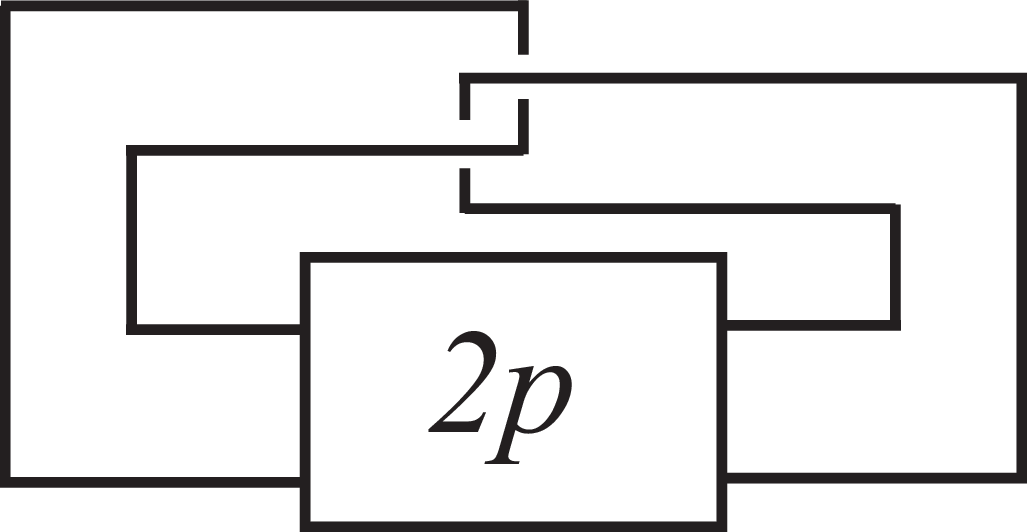}
 \end{center}
\caption{Here $2p \,  \, (p\in \Z \setminus \{0,1\})$ denote the numbers of half twists in each box.}
\label{1006}
\end{figure}

\section{Lemmas necessary for Proof of Volume Conjecture for twist knots}
Let $N$ be a positive integer, $n_i$ be positive integers or equal to $0$, and
$q$ be ${\rm exp}(2 \pi \sqrt{-1}/N)$.
We know the following theorem.
\begin{thm}[{\cite{11}}]\label{155}
Let $p$ be a positive integer, $K_{p>0}$ be the usual $p$-th twist knot, $J_N(K;q)$ be the colored Jones polynomial of $K$, and $K^{*}$ be the mirror image of the knot $K$. 
We know that $K_{-p}$ is the mirror image of $K_{p>0}$.
\begin{align*}
J_N(K^{*}_{p>0};q)=&q^{1-N}\sum_{N-1\geq n_{2p-1}\geq \cdots \geq n_1 \geq 0}(q^{1-N})_{n_{2p-1}}q^{-N n_{2p-1}}\\
&\times \prod_{i=1}^{2p-2}(-1)^{n_i} q^{(-1)^{i}N n_i+\binom{n_i}{2}
-n_in_{i+1}}{\begin{bmatrix}n_{i+1}\\n_i\end{bmatrix}},
\end{align*}
where we have used the usual $q$-binomial coefficient
$$
{\begin{bmatrix}n_{i+1}\\n_i\end{bmatrix}}=\frac{(q)_{n_{i+1}}}{(q)_{n_{i+1}-n_i}(q)_{n_{i}}}
$$
with the standard $q$-hyper geometric notation
$$
(q)_n=\prod_{k=0}^{n-1}(1-q^{k+1})=(1-q)^n\Gamma_q(n+1).
$$
\end{thm}

\begin{defn}We say that $f(N)$ and $g(N)$ are {\it equivalent}
if the difference $|f(N)-g(N)|$ tends to zero when $N \to \infty$. Namely, we denote 
$\displaystyle \lim_{N\to \infty}|f(N)-g(N)|=0$ by $f(N)\overset{N}{\sim} g(N)$.
\end{defn}

This relation $\overset{N}{\sim}$ is an equivalence relation.
Hence, reflexive relation, symmetric relation, and transitive relation hold.

\begin{lem}\label{99}
If $i(N)\overset{N}{\sim} g(N)$, $h(N) \overset{N}{\sim} i(N)$, and $g(N)\leq f(N)\leq h(N)$,
then we have that
$$
f(N)\overset{N}{\sim} i(N).
$$
Moreover, if $i(N)$ is a constant function, then
$$
\lim_{N\to \infty}f(N)=i(N)
$$
holds.
\end{lem}
\begin{proof}
Because of the assumptions, we have that
$$
g(N)-i(N)\leq f(N)-i(N)\leq h(N)-i(N),
$$
and
$$
\lim_{N\to \infty}|g(N)-i(N)|=\lim_{N\to \infty}|h(N)-i(N)|=0.
$$
Hence,
$$
\lim_{N\to \infty}|f(N)-i(N)|=0,
$$
and
$$
f(N)\overset{N}{\sim} i(N).
$$
\end{proof}
%We easily see that if there exist $M>0$ such that for every $n\in [0,N]$ and for every $N'\in \N$,
%$f(n,N')<M$, then we obtain the following formula:
%$$
%\lim_{N'\to \infty}\int_0^N f(n,N')dn=\int_0^N\lim_{N'\to \infty}f(n,N')dn.
%$$
%Hence,
%$$
%\int_0^N f(n,N')dn \overset{N'}{\sim} \int_0^N \lim_{N'\to \infty}f(n,N')dn. 
%$$
%We easily see that if $f_N(x)-g_N(x)$ is uniformly convergence,
%then we obtain the following formula:
%$$
%f_N(x)\overset{N}{\sim} g_N(x) \Longrightarrow \int_a^b f_N(x)dx \overset{N}{\sim} \int_a^b g_N(x)dx.
%$$
%Moreover, we easily see that if $f(N)\in \R_{>1}$ and $g(N)\in \R_{>1}$, then we obtain the following formula:
%$$
%f(N)\overset{N}{\sim} g(N) \Longrightarrow \log f(N) \overset{N}{\sim} \log g(N).
%$$
%For every $C\in \R_{>0}$,
%we easily see that if $|g(N)|<C$, then we obtain the following formula:
%$$
%\log f(N) \overset{N}{\sim} \log g(N) \Longrightarrow f(N) \overset{N}{\sim} g(N).
%$$
%
%\begin{defn}
%We denote a formula with $\overset{N}{\sim}$ and $\leq$ by $\overset{N}{\lesssim}$,
%and similarly a formula with $\overset{N}{\sim}$ and $\geq$ by $\overset{N}{\gtrsim}$ ;
%For example $g_1(N)\overset{N}{\sim} g_2(N)\leq g_3(N)\leq g_4(N)\overset{N}{\sim} g_5(N)\overset{N}{\sim}
%g_6(N)\leq f(N)$ tends to $g_1(N)\overset{N}{\lesssim} f(N)$.
%\end{defn}
%Here, $\log z:=\log |z|+\sqrt{-1}{\rm Arg}(z) \ \ (0\leq {\rm Arg}(z)<2\pi)$.
We define the following dilogarithm function:
$$
{\rm Li}_2(z):=\sum_{n=1}^{\infty}\frac{z^n}{n^2} \, \, (|z|<1).
$$
We know that the following equality:
$$
-\log (1-z)=\sum_{n=1}^{\infty}\frac{z^n}{n} \, \, (|z|<1).
$$
Hence, the analytic continuation of the dilogarithm function is given by
$$
{\rm Li}_2(z)=-\int_{0}^{z} \frac{\log (1-t)}{t} dt \ \ (z\in \C \setminus [1, \infty)).
$$
Here, $\log (1-t):=\log |1-t|+\sqrt{-1}{\rm Arg}(1-t) \ \ (0\leq {\rm Arg}(1-t)<2 \pi)$ 
is a regular analytic function on $\C\setminus [1, \infty)$.
Hence, ${\rm Li}_2(z)$ is a regular analytic function on $\C\setminus [1, \infty)$.
\begin{defn}\label{91}
We define the following functions:
$$
f(z_1,\ldots,z_{2p-1}):=-{\rm Li}_2(\frac{1}{z_1})+\sum_{i=1}^{2p-2}\Bigl({\rm Li}_2(z_i)
-{\rm Li}_2(\frac{z_i}{z_{i+1}})+{\rm Li}_2(\frac{1}{z_{i+1}})\Bigr)+{\rm Li}_2(z_{2p-1}),
$$
$$
f(z^{(1)}_1,\ldots ,z^{(2p-1)}_{2p-1}):=\sum_{i=1}^{2p-2}\Bigl({\rm Li}_2(z^{(i)}_i)
-{\rm Li}_2(\frac{z^{(i)}_i}{z^{(i)}_{i+1}})+{\rm Li}_2(\frac{1}{z^{(i)}_{i+1}})\Bigr)-{\rm Li}_2(z^{(2p-1)}_{2p-1})
+\frac{\pi^2}{6},
$$
\begin{align*}
\tilde{f}(z^{(1)}_1,\ldots z^{(2p-1)}_{2p-1}):=&\sum_{i=1}^{2p-2}
\Bigl(-\sqrt{-1}{\rm Li}_2(z^{(i)}_i)+\frac{\pi^2 \sqrt{-1}}{6}
+\sqrt{-1}{\rm Li}_2(\frac{z^{(i)}_i}{z^{(i)}_{i+1}})-\frac{\pi^2 \sqrt{-1}}{6}\\
&-\sqrt{-1}{\rm Li}_2(\frac{1}{z^{(i)}_{i+1}})+\frac{\pi^2 \sqrt{-1}}{6}\Bigr)
+\sqrt{-1}{\rm Li}_2(z^{(2p-1)}_{2p-1})-\frac{\pi^2 \sqrt{-1}}{6},
\end{align*}
and
$$
g(n,N):=\frac{\pi^2 \sqrt{-1}n^2}{N^2}-\frac{2\pi n\log (1-\exp (2\pi \sqrt{-1} n/N))}{N}+\frac{\pi n \log 4}{N}
+\frac{\pi n \sin ^2(\frac{\pi n}{N})}{N},
$$
where $z^{(i)}_i:={\rm exp}(2 \pi \sqrt{-1} n_i/N), \, \, 1/z^{(i)}_{i+1}:={\rm exp}(2 \pi \sqrt{-1} (n_{i+1}-n_i)/N)$.
\end{defn}
Hence,
$$
\tilde{f}(z^{(1)}_1,\ldots z^{(2p-1)}_{2p-1})=-\sqrt{-1}f(z^{(1)}_1,\ldots z^{(2p-1)}_{2p-1})+\frac{\pi^2 \sqrt{-1}
(2p-2)}{6}.
$$

\subsection{First reduction}
Firstly, we prove that
$$
f(z^{(1)}_1,\ldots,z^{(2p-1)}_{2p-1})\overset{N}{\sim} f({\rm exp}(2 \pi \sqrt{-1} n_1/N),\ldots,{\rm exp}(2 \pi \sqrt{-1} n_{2p-1}/N)),
$$
$$
f(z^{(1)}_1,\ldots,z^{(2p-1)}_{2p-1})+{\rm Li}_2(z_{2p-1}^{(2p-1)})-\frac{\pi^2}{6}
\overset{N}{\sim} f({\rm exp}(2 \pi \sqrt{-1} n_1/N),\ldots,{\rm exp}(2 \pi \sqrt{-1} n_{2p-1}/N)),
$$
$$
\frac{2 \pi \log |(q)_n|}{N} \overset{N}{\sim}
\sqrt{-1}{\rm Li}_2({\rm exp}(2 \pi \sqrt{-1} n/N))-\frac{\pi^2\sqrt{-1}}{6}
\overset{N}{\sim} {\rm Im}[- {\rm Li}_2({\rm exp}(2 \pi \sqrt{-1} n/N))],
$$
and the following limit value converges:
$$
\lim_{N \to \infty} {\rm Li}_2({\rm exp}(2 \pi \sqrt{-1} n/N)).
$$
Hence, the following limit value converges:
$$
\lim_{N\to \infty}f({\rm exp}(2 \pi \sqrt{-1} n_1/N),\ldots,{\rm exp}(2 \pi \sqrt{-1} n_{2p-1}/N)).
$$
Moreover, we obtain the following formulas:
\begin{align*}
&\frac{2 \pi}{N}\log \Bigl(|(q)_{n_{2p-1}}|\prod_{i=1}^{2p-2} \Bigl|{\begin{bmatrix}n_{i+1}\\n_i\end{bmatrix}}
\Bigr|\Bigr)
= \frac{2 \pi}{N}\Bigl(\log |(q)_{n_{2p-1}}|+\sum_{i=1}^{2p-2} \log \Bigl|{\begin{bmatrix}n_{i+1}\\n_i\end{bmatrix}}\Bigr|\Bigr)\\
&\overset{N}{\sim}
\sum_{i=1}^{2p-2} {\rm Im}[{\rm Li}_2({\rm exp}(2 \pi \sqrt{-1} n_i/N))-{\rm Li}_2({\rm exp}(2 \pi \sqrt{-1} n_{i+1}/N))\\
& \ \ +{\rm Li}_2({\rm exp}(2 \pi \sqrt{-1} (n_{i+1}-n_i)/N))]
+{\rm Im}[- {\rm Li}_2({\rm exp}(2 \pi \sqrt{-1} n_{2p-1}/N))]\\
&={\rm Im}[f(z^{(1)}_1,\ldots ,z^{(2p-1)}_{2p-1})]\\ 
&\overset{N}{\sim} {\rm Im}[f({\rm exp}(2 \pi \sqrt{-1} n_1/N),\ldots ,{\rm exp}(2 \pi \sqrt{-1} n_{2p-1}/N))],
\end{align*}
and
\begin{align*}
&\frac{2 \pi}{N}\log \Bigl(|(q)_{n_{2p-1}}|\prod_{i=1}^{2p-2} \Bigl|{\begin{bmatrix}n_{i+1}\\n_i\end{bmatrix}}
\Bigr|\Bigr)
= \frac{2 \pi}{N}\Bigl(\log |(q)_{n_{2p-1}}|+\sum_{i=1}^{2p-2} \log \Bigl|{\begin{bmatrix}n_{i+1}\\n_i\end{bmatrix}}\Bigr|\Bigr)\\
&\overset{N}{\sim}\tilde{f}(z^{(1)}_1,\ldots ,z^{(2p-1)}_{2p-1})\\
&=-\sqrt{-1}f(z^{(1)}_1,\ldots ,z^{(2p-1)}_{2p-1})+\frac{\pi^2\sqrt{-1}(2p-2)}{6}\\
&\overset{N}{\sim}
-\sqrt{-1}f({\rm exp}(2 \pi \sqrt{-1} n_1/N),\ldots ,{\rm exp}(2 \pi \sqrt{-1} n_{2p-1}/N))
+\frac{\pi^2\sqrt{-1}(2p-2)}{6}.
\end{align*}
To show the above formulas, we first prove the property of dilogarithm function ${\rm Li}_2(z)$
of Lemma \ref{84}, Lemma \ref{85} and Lemma \ref{100}.

\begin{lem}\label{84}
If $z\in \C \setminus (-\infty,0]$, then we obtain the following functional equality:
$$
{\rm Li}_2(1-z)=-\frac{1}{2}\log ^2 z-{\rm Li}_2(1-\frac{1}{z}).
$$
\end{lem}
\begin{proof}
We obtain the following equalities:
\begin{align*}
{\rm Li}_2(1-z)+{\rm Li}_2(1-\frac{1}{z})&={\rm Li}_2(1-z)-\int_0^{1-1/z}\frac{1-t}{t}dt\\
&={\rm Li}_2(1-z)-\int_1^z\frac{\log (\frac{1}{t})}{1-\frac{1}{t}}\cdot \frac{dt}{t^2}\\
&={\rm Li}_2(1-z)-\int_1^z\frac{\log t}{t} dt-\int_1^z \frac{\log t}{t-1}dt\\
&={\rm Li}_2(1-z)-\Bigl[\frac{1}{2}\log ^2 t\Bigr]_1^z+\int_0^{1-z}\frac{1-t}{t} dt\\
&={\rm Li}_2(1-z)-\frac{1}{2}\log ^2 z-{\rm Li}_2(1-z)
=-\frac{1}{2}\log^2 z.
\end{align*}
\end{proof}
\begin{lem}\label{85}
If $z\in \C \setminus [0, \infty)$, then we obtain the following functional equality:
$$
-{\rm Li}_2(\frac{1}{z})={\rm Li}_2(z)+\frac{\pi^2}{6}+\frac{1}{2}\log^2 (-z).
$$
\end{lem}
\begin{proof}We obtain the following equalities:
\begin{align*}
{\rm Li}_2(z)+{\rm Li}_2(\frac{1}{z})&=\int_{z}^{0}\frac{\log (1-t)}{t} dt+\int_{1/z}^{0}\frac{\log (1-t)}{t} dt\\
&=\int_{z}^{-1}\frac{\log (1-t)}{t} dt+\int_{1/z}^{-1}\frac{\log (1-t)}{t} dt+2\int_{-1}^{0}\frac{\log (1-t)}{t} dt\\
&=2 {\rm Li}_2(-1)+\int_{z}^{-1}\frac{\log (1-t)}{t} dt+\int_{z}^{-1}t \log (\frac{t-1}{t}) \frac{-dt}{t^2}\\
&=-\frac{\pi^2}{6}+\int_{z}^{-1}\frac{\log (1-t)}{t}-\frac{\log (1-t)}{t} +\frac{\log (-t)}{t} dt\\
&=-\frac{\pi^2}{6}-\frac{1}{2}\log^2 (-z).
\end{align*}
\end{proof}
\begin{lem}\label{100}
If the domain of definition of ${\rm Li}_2(z)$ is $z\in \C \setminus [1,\infty)$
and that of  ${\rm Li}_2(-z)$ is $z\in \C \setminus (-\infty,-1]$, 
then we obtain the following functional equality:
$$
{\rm Li}_2(z)+{\rm Li}_2(-z)=\frac{1}{2} {\rm Li}_2(z^2).
$$
\end{lem}
\begin{proof}
We obtain the following equalities:
\begin{align*}
{\rm Li}_2(z)+{\rm Li}_2(-z)&=\sum_{n=1}^{\infty}\Bigl(\frac{z^n}{n^2}+\frac{(-z)^n}{n^2}\Bigr)\\
&=\sum_{m=1}^{\infty}\Bigl(\frac{z^{2m-1}}{(2m-1)^2}-\frac{z^{2m-1}}{(2m-1)^2}\Bigr)+\sum_{k=1}^{\infty}
\Bigl(\frac{z^{2k}}{(2k)^2}+\frac{z^{2k}}{(2k)^2}\Bigr)\\
&=\frac{1}{2}{\rm Li}_2(z^2).
\end{align*}
\end{proof}

\begin{lem}\label{98}
The following limit value uniformly converges: for every $n\in \C$,
$$
\lim_{N \to \infty} {\rm Li}_2({\rm exp}(2 \pi \sqrt{-1} n/N)).
$$
\end{lem}
\begin{proof}
For sufficiently large $N$ and $N'$ such that $N>N'$, we obtain the following inequalities:
\begin{align*}
&|{\rm Li}_2({\rm exp}(2 \pi \sqrt{-1} n/N))-{\rm Li}_2({\rm exp}(2 \pi \sqrt{-1} n/N'))|\\
&=\Bigl|\int_{{\rm exp}(2 \pi \sqrt{-1} n/N')}^{{\rm exp}(2 \pi \sqrt{-1} n/N)}\frac{\log (1-t)}{t} dt\Bigr|\\
&\leq \int_{\frac{2 \pi n}{N}}^{\frac{2 \pi n}{N'}}\Bigl|\frac{\log(1-{\rm exp}(\sqrt{-1} \theta))}{{\rm exp}(\sqrt{-1} \theta)}\Bigr||\sqrt{-1}{\rm exp}(\sqrt{-1} \theta)|d\theta \\
&=\int_{\frac{2 \pi n}{N}}^{\frac{2 \pi n}{N'}}|\log(1-{\rm exp}(\sqrt{-1}\theta)|d\theta\\
&\leq \int_{\frac{2 \pi n}{N}}^{\frac{2 \pi n}{N'}} (1-\log |\theta|) d\theta
=\Bigl[2\theta -\theta \log |\theta| \Bigr]_{\frac{2 \pi n}{N}}^{\frac{2 \pi n}{N'}}\\
&=\frac{2\pi n}{N'}\Bigl(2-\log \Bigl|\frac{2\pi n}{N'}\Bigr|\Bigr)
-\frac{2\pi n}{N}\Bigl(2-\log \Bigl|\frac{2\pi n}{N}\Bigl|\Bigr)
 \to 0 \ \ (N>N'\to \infty).
\end{align*}
\end{proof}

\begin{lem}\label{101}
With the same notation as in Theorem \ref{155} we have that
\begin{align*}
\frac{2 \pi \log |(q)_n|}{N}&=
\sqrt{-1}{\rm Li}_2({\rm exp}(2 \pi \sqrt{-1} n/N))-\frac{\pi^2\sqrt{-1}}{6}+g(n,N)\\
&\overset{N}{\sim}{\rm Im}[- {\rm Li}_2({\rm exp}(2 \pi \sqrt{-1} n/N))].
\end{align*}
\end{lem}
\begin{proof}
Because of Lemma \ref{84} and Lemma \ref{85},
we obtain the following equalities:
\begin{align*}
&\sqrt{-1}{\rm Li}_2({\rm exp}(2 \pi \sqrt{-1}n/N))-\frac{\pi^2 \sqrt{-1}}{6}
-{\rm Im}[- {\rm Li}_2({\rm exp}(2 \pi \sqrt{-1} n/N))]\\
&=\sqrt{-1}({\rm Re}[{\rm Li}_2({\rm exp}(2 \pi \sqrt{-1}n/N))]
+\sqrt{-1}{\rm Im}[{\rm Li}_2({\rm exp}(2 \pi \sqrt{-1}n/N))])\\
& \ \ -\frac{\pi^2 \sqrt{-1}}{6}
+{\rm Im}[ {\rm Li}_2({\rm exp}(2 \pi \sqrt{-1} n/N))]\\
&=\sqrt{-1}({\rm Re}[{\rm Li}_2({\rm exp}(2 \pi \sqrt{-1}n/N))]-\frac{\pi^2}{6})\\
&=\sqrt{-1}({\rm Re}[-\frac{1}{2}\log ^2 (1-\exp (2 \pi \sqrt{-1}n/N))\\
& \ \ -{\rm Li}_2\Bigl(1-\frac{1}{1-\exp (2\pi \sqrt{-1}n/N)}\Bigr)]-\frac{\pi^2}{6})\\
&=\sqrt{-1}({\rm Re}[-\frac{1}{2}\log ^2 (1-\exp (2 \pi \sqrt{-1}n/N))\\
& \ \ -{\rm Li}_2\Bigl(\frac{\exp (2\pi \sqrt{-1}n/N)}{\exp (2\pi \sqrt{-1}n/N)-1}\Bigr)]-\frac{\pi^2}{6})\\
&=\sqrt{-1}({\rm Re}[-\frac{1}{2}\log ^2 (1-\exp (2 \pi \sqrt{-1}n/N))+
{\rm Li}_2\Bigl(\frac{\exp (2\pi \sqrt{-1}n/N)-1}{\exp (2\pi \sqrt{-1}n/N)}\Bigr)\\
& \ \ +\frac{\pi^2}{6}
+\frac{1}{2}\log ^2 \Bigl(-\frac{\exp (2\pi \sqrt{-1}n/N)}{\exp (2\pi \sqrt{-1}n/N)-1}\Bigr)]-\frac{\pi^2}{6})\\
&=\sqrt{-1}({\rm Re}[{\rm Li}_2\Bigl(\frac{\exp (2\pi \sqrt{-1}n/N)-1}{\exp (2\pi \sqrt{-1}n/N)}\Bigr)\\
& \ \ -\log (\exp (2\pi \sqrt{-1}n/N))\log(1-\exp (2\pi \sqrt{-1}n/N))\\
& \ \ +\frac{1}{2}\log ^2 (\exp (2\pi \sqrt{-1}n/N))]) \to 0 \ \ (N\to \infty).
\end{align*}
Hence,
\begin{align*}
\frac{2 \pi \log |(q)_n|}{N}&=\frac{2 \pi \log |\prod_{k=0}^{n-1}(1-q^{k+1})|}{N}
=\frac{2 \pi \sum_{k=1}^{n} \log |1-q^k|}{N}\\
&\overset{N}{\sim} 2 \pi \int_0^{n/N} \log |1-{\rm exp}(2 \pi \sqrt{-1} x)| dx \\
&=\sqrt{-1}{\rm Li}_2({\rm exp}(2 \pi \sqrt{-1}n/N))+\frac{\pi^2 \sqrt{-1} n^2}{N^2}\\
& \ \ -\frac{2 \pi n \log (1-{\rm exp}(2 \pi \sqrt{-1} n/N))}{N}\\
& \ \ +\frac{\pi n \log 4}{N}
+\frac{\pi n \log \sin ^2(\frac{\pi n}{N})}{N}-\frac{\pi^2 \sqrt{-1}}{6}\\  
&= \sqrt{-1}{\rm Li}_2({\rm exp}(2 \pi \sqrt{-1}n/N))-\frac{\pi^2\sqrt{-1}}{6}+g(n,N)\\
&\overset{N}{\sim} {\rm Im}[- {\rm Li}_2({\rm exp}(2 \pi \sqrt{-1} n/N))].
\end{align*}
\end{proof}
\begin{cor}\label{102}
With the same notation as in Theorem \ref{155} we have that
\begin{align*}
\frac{2 \pi \log \Bigl|{\begin{bmatrix}n_{i+1}\\n_i\end{bmatrix}}\Bigr|}{N}
&\overset{N}{\sim}
-\sqrt{-1}{\rm Li}_2({\rm exp}(2 \pi \sqrt{-1} n_i/N))+\frac{\pi^2\sqrt{-1}}{6}-g(n_i,N)\\
& \ \ +\sqrt{-1}{\rm Li}_2({\rm exp}(2 \pi \sqrt{-1} n_{i+1}/N))-\frac{\pi^2\sqrt{-1}}{6}+g(n_{i+1},N)\\
& \ \ -\sqrt{-1}{\rm Li}_2({\rm exp}(2 \pi \sqrt{-1} (n_{i+1}-n_i)/N))+\frac{\pi^2 \sqrt{-1}}{6}\\
& \ \ -g(n_{i+1}-n_i,N)\\
&\overset{N}{\sim}
{\rm Im}[{\rm Li}_2({\rm exp}(2 \pi \sqrt{-1} n_i/N))-{\rm Li}_2({\rm exp}(2 \pi \sqrt{-1} n_{i+1}/N))\\
& \ \ +{\rm Li}_2({\rm exp}(2 \pi \sqrt{-1} (n_{i+1}-n_i)/N))].
\end{align*}
\end{cor}

\begin{lem}\label{103}
$f(z^{(1)}_1,\ldots,z^{(2p-1)}_{2p-1})$ and $f(z_1,\ldots,z_{2p-1})$ in Definition \ref{91} above. 
We obtain the following equivalence relations:
$$
f(z^{(1)}_1,\ldots,z^{(2p-1)}_{2p-1})
\overset{N}{\sim} f({\rm exp}(2 \pi \sqrt{-1} n_1/N),\ldots,{\rm exp}(2 \pi \sqrt{-1} n_{2p-1}/N)),
$$
and
$$
f(z^{(1)}_1,\ldots,z^{(2p-1)}_{2p-1})+{\rm Li}_2(z_{2p-1}^{(2p-1)})-\frac{\pi^2}{6}
\overset{N}{\sim} f({\rm exp}(2 \pi \sqrt{-1} n_1/N),\ldots,{\rm exp}(2 \pi \sqrt{-1} n_{2p-1}/N)).
$$
\end{lem}
\begin{proof}
Firstly, for a sufficiently large $N$,
\begin{align*}
&\Bigl| -{\rm Li}_2(\frac{z^{(i)}_i}{z^{(i)}_{i+1}})
+{\rm Li}_2(\frac{{\rm exp}(2\pi \sqrt{-1}n_i/N)}{{\rm exp}(2\pi \sqrt{-1}n_{i+1}/N)})\Bigr|\\
&=|-{\rm Li}_2({\rm exp}(2 \pi \sqrt{-1}n_{i+1})/N))
+{\rm Li}_2({\rm exp}(2 \pi \sqrt{-1}(n_i-n_{i+1})/N))|\\
&\leq \int_{\frac{2\pi (n_i-n_{i+1})}{N}}^{\frac{2 \pi n_{i+1}}{N}}|\log (1-{\rm exp}(\sqrt{-1}\theta))|
d\theta \\
&\leq \int_{\frac{2\pi (n_i-n_{i+1})}{N}}^{\frac{2 \pi n_{i+1}}{N}} (1-\log |\theta |) d \theta
= \Bigl[2\theta-\theta \log |\theta |\Bigr]_{\frac{2\pi (n_i-n_{i+1})}{N}}^{\frac{2 \pi n_{i+1}}{N}}\\
&=\frac{2 \pi n_{i+1}}{N}(2-\log \frac{2\pi n_{i+1}}{N})
-\frac{2 \pi (n_i-n_{i+1})}{N}(2-\log \frac{2\pi (n_{i+1}-n_i)}{N})\\
&\to 0 \ \ (N\to \infty).
\end{align*}
Secondly, for a sufficiently large $N$,
\begin{align*}
&\Bigl| {\rm Li}_2(\frac{1}{z^{(i)}_{i+1}})-{\rm Li}_2(\frac{1}{{\rm exp}(2 \pi \sqrt{-1} n_i/N)})\Bigr|\\
&=|{\rm Li}_2({\rm exp}(2 \pi \sqrt{-1}n_{i+1}/N))-{\rm Li}_2({\rm exp}(-2 \pi \sqrt{-1} n_{i+1}/N)|\\
&\leq \int_{\frac{-2\pi n_{i+1}}{N}}^{\frac{2 \pi n_{i+1}}{N}}|\log (1-{\rm exp}(\sqrt{-1}\theta))|
d\theta\\
&\leq \int_{\frac{-2\pi n_{i+1}}{N}}^{\frac{2 \pi n_{i+1}}{N}}(1-\log |\theta |) d \theta
=\Bigl[2\theta-\theta \log |\theta |\Bigr]_{\frac{-2\pi n_{i+1}}{N}}^{\frac{2 \pi n_{i+1}}{N}}\\
&=\frac{2 \pi n_{i+1}}{N}(2-\log \frac{2\pi n_{i+1}}{N})
+\frac{2 \pi n_{i+1}}{N}(2-\log \frac{2\pi n_{i+1}}{N})\\
&\to 0 \ \ (N\to \infty).
\end{align*}
Thirdly, for a sufficiently large $N$, because of Lemma \ref{100},
\begin{align*}
&\Bigl|-{\rm Li}_2(z^{(2p-1)}_{2p-1})-{\rm Li}_2({\rm exp}(2 \pi \sqrt{-1}n_{2p-1}/N))
+{\rm Li}_2(\frac{1}{{\rm exp}(2 \pi \sqrt{-1} n_1/N)})+\frac{\pi^2}{6}\Bigr|\\
&=\Bigl|2{\rm Li}_2({\rm exp}(2 \pi \sqrt{-1}n_{2p-1}/N))-{\rm Li}_2({\rm exp}(-2\pi \sqrt{-1}n_1/N))-\frac{\pi^2}{6}\Bigr|\\
&=\Bigl|{\rm Li}_2({\rm exp}(4 \pi \sqrt{-1}n_{2p-1}/N))-2{\rm Li}_2(-{\rm exp}(2\pi \sqrt{-1}n_{2p-1}/N))\\
& \ \ -{\rm Li}_2({\rm exp}(-2\pi \sqrt{-1}n_1/N))-\frac{\pi^2}{6}\Bigr|\\
&\leq \int_{\frac{-2\pi n_1}{N}}^{\frac{4 \pi n_{2p-1}}{N}}|\log (1-{\rm exp}(\sqrt{-1}\theta))| d\theta
+|-2{\rm Li}_2(-{\rm exp}(2 \pi \sqrt{-1}n_{2p-1}/N))-\frac{\pi^2}{6}|\\
&\leq \int_{\frac{-2\pi n_1}{N}}^{\frac{4 \pi n_{2p-1}}{N}}(1-\log |\theta |) d \theta
+|-2{\rm Li}_2(-{\rm exp}(2 \pi \sqrt{-1}n_{2p-1}/N))-\frac{\pi^2}{6}|\\
&=\Bigl[2\theta-\theta \log |\theta |\Bigr]_{\frac{-2\pi n_1}{N}}^{\frac{4 \pi n_{2p-1}}{N}}
+|-2{\rm Li}_2(-{\rm exp}(2 \pi \sqrt{-1}n_{2p-1}/N))-\frac{\pi^2}{6}|\\
&=\frac{4 \pi n_{2p-1}}{N}(2-\log \frac{4\pi n_{2p-1}}{N})
+\frac{2 \pi n_1}{N}(2-\log \frac{2\pi n_{1}}{N})\\
& \ \ +|-2{\rm Li}_2(-{\rm exp}(2 \pi \sqrt{-1}n_{2p-1}/N))-\frac{\pi^2}{6}|\\
&\to 0 \ \ (N\to \infty).
\end{align*}
Fourthly, for a sufficiently large $N$,
\begin{align*}
&\Bigl|-{\rm Li}_2({\rm exp}(2 \pi \sqrt{-1}n_{2p-1}/N))
+{\rm Li}_2(\frac{1}{{\rm exp}(2 \pi \sqrt{-1} n_1/N)})\Bigr|\\
&=\Bigl|{\rm Li}_2({\rm exp}(2 \pi \sqrt{-1}n_{2p-1}/N))-{\rm Li}_2({\rm exp}(-2\pi \sqrt{-1}n_1/N))\Bigr|\\
&\leq \Bigl|\int_{\frac{-2\pi n_1}{N}}^{\frac{2 \pi n_{2p-1}}{N}}\log (1-{\rm exp}(\sqrt{-1}\theta)) d\theta
\Bigr|
\leq \int_{\frac{-2\pi n_1}{N}}^{\frac{2 \pi n_{2p-1}}{N}}|\log (1-{\rm exp}(\sqrt{-1}\theta))| d\theta \\
&\leq \int_{\frac{-2\pi n_1}{N}}^{\frac{2 \pi n_{2p-1}}{N}}(1-\log |\theta |) d \theta 
=\Bigl[2\theta-\theta \log |\theta |\Bigr]_{\frac{-2\pi n_1}{N}}^{\frac{2 \pi n_{2p-1}}{N}}\\
&=\frac{2 \pi n_{2p-1}}{N}(2-\log \frac{2\pi n_{2p-1}}{N})
+\frac{2 \pi n_1}{N}(2-\log \frac{2\pi n_{1}}{N})\\
&\to 0 \ \ (N\to \infty).
\end{align*}
Hence,
\begin{align*}
&|f(z^{(1)}_1,\ldots,z^{(2p-1)}_{2p-1})-f({\rm exp}(2 \pi \sqrt{-1} n_1/N),\ldots,{\rm exp}(2 \pi \sqrt{-1} n_{2p-1}/N))|\\
&=\Bigl|\sum_{i=1}^{2p-2}\Bigl({\rm Li}_2(z^{(i)}_i)-{\rm Li}_2({\rm exp}(2 \pi \sqrt{-1} n_i/N))
-{\rm Li}_2(\frac{z^{(i)}_i}{z^{(i)}_{i+1}})\\
& \ \ +{\rm Li}_2(\frac{{\rm exp}(2 \pi \sqrt{-1} n_i/N)}
{{\rm exp}(2 \pi \sqrt{-1} n_{i+1}/N)})
+{\rm Li}_2(\frac{1}{z^{(i)}_{i+1}})-{\rm Li}_2(\frac{1}{{\rm exp}(2 \pi \sqrt{-1} n_i/N)})\Bigr)\\
& \ \ -{\rm Li}_2(z^{(2p-1)}_{2p-1})-{\rm Li}_2({\rm exp}(2 \pi \sqrt{-1} n_{2p-1}/N))
+{\rm Li}_2(\frac{1}{{\rm exp}(2 \pi \sqrt{-1} n_1/N)})+\frac{\pi^2}{6}\Bigr|\\
&\leq \sum_{i=1}^{2p-2}\Bigl(|{\rm Li}_2(z^{(i)}_i)-{\rm Li}_2({\rm exp}(2 \pi \sqrt{-1} n_i/N))|\\
& \ \ +|-{\rm Li}_2(\frac{z^{(i)}_i}{z^{(i)}_{i+1}})+{\rm Li}_2(\frac{{\rm exp}(2 \pi \sqrt{-1} n_i/N)}
{{\rm exp}(2 \pi \sqrt{-1} n_{i+1}/N)})|\\
& \ \ +|{\rm Li}_2(\frac{1}{z^{(i)}_{i+1}})-{\rm Li}_2(\frac{1}{{\rm exp}(2 \pi \sqrt{-1} n_i/N)})|\Bigr)\\
& \ \ +|-{\rm Li}_2(z^{(2p-1)}_{2p-1})-{\rm Li}_2({\rm exp}(2 \pi \sqrt{-1} n_{2p-1}/N))\\
& \ \ +{\rm Li}_2(\frac{1}{{\rm exp}(2 \pi \sqrt{-1} n_1/N)})+\frac{\pi^2}{6}| \to 0 \, \, (N\to \infty),
\end{align*}
and
\begin{align*}
&|f(z^{(1)}_1,\ldots,z^{(2p-1)}_{2p-1})+{\rm Li}_2(z^{(2p-1)}_{2p-1})-\frac{\pi^2}{6}\\
& \ \ -f({\rm exp}(2 \pi \sqrt{-1} n_1/N),\ldots,{\rm exp}(2 \pi \sqrt{-1} n_{2p-1}/N))|\\
&=\Bigl|\sum_{i=1}^{2p-2}\Bigl({\rm Li}_2(z^{(i)}_i)-{\rm Li}_2({\rm exp}(2 \pi \sqrt{-1} n_i/N))
-{\rm Li}_2(\frac{z^{(i)}_i}{z^{(i)}_{i+1}})\\
& \ \ +{\rm Li}_2(\frac{{\rm exp}(2 \pi \sqrt{-1} n_i/N)}
{{\rm exp}(2 \pi \sqrt{-1} n_{i+1}/N)})
+{\rm Li}_2(\frac{1}{z^{(i)}_{i+1}})-{\rm Li}_2(\frac{1}{{\rm exp}(2 \pi \sqrt{-1} n_i/N)})\Bigr)\\
& \ \ -{\rm Li}_2({\rm exp}(2 \pi \sqrt{-1} n_{2p-1}/N))
+{\rm Li}_2(\frac{1}{{\rm exp}(2 \pi \sqrt{-1} n_1/N)})\Bigr|\\
&\leq \sum_{i=1}^{2p-2}\Bigl(|{\rm Li}_2(z^{(i)}_i)-{\rm Li}_2({\rm exp}(2 \pi \sqrt{-1} n_i/N))|\\
& \ \ +|-{\rm Li}_2(\frac{z^{(i)}_i}{z^{(i)}_{i+1}})+{\rm Li}_2(\frac{{\rm exp}(2 \pi \sqrt{-1} n_i/N)}
{{\rm exp}(2 \pi \sqrt{-1} n_{i+1}/N)})|\\
& \ \ +|{\rm Li}_2(\frac{1}{z^{(i)}_{i+1}})-{\rm Li}_2(\frac{1}{{\rm exp}(2 \pi \sqrt{-1} n_i/N)})|\Bigr)\\
& \ \ 
+|-{\rm Li}_2({\rm exp}(2 \pi \sqrt{-1} n_{2p-1}/N))+{\rm Li}_2(\frac{1}{{\rm exp}(2 \pi \sqrt{-1} n_1/N)})|\\
&\to 0 \, \, (N\to \infty).
\end{align*}
\end{proof}

\subsection{Second reduction}

Let $(a_1,\ldots,a_{2p-1})\in (\C \setminus [0,\infty))^{2p-1}$ be a solution of $\partial f/\partial z_1 =\cdots =\partial f/\partial z_{2p-1}=0$. The solution of these equations which induces the complete hyperbolic
structure of the knot complement exists uniquely by virtue of the Lemma $2.3$ of \cite{17}.
Secondly, we prove the following equality by using the method of steepest descent:
\begin{align*}
&\int_{0}^N\cdots \int_{0}^N
\exp \Bigl(\frac{\sqrt{-1}N}{2\pi}(f(\exp (2\pi\sqrt{-1}n_1/N),\ldots ,\exp (2\pi \sqrt{-1}n_{2p-1}/N))\\
& \ \ -\frac{\pi^2(2p-2)}{6})\Bigr)dn_1\cdots dn_{2p-1}\\
&=
\exp \Bigl(-\frac{3\pi\sqrt{-1} (2p-1)}{4}\Bigr) \cdot N^{\frac{2p-1}{2}}
\cdot \exp \Bigl(-\frac{N}{2\pi}{\rm Im}[f(a_1,\ldots ,a_{2p-1})]\Bigr)\\
& \ \ \times \Bigl(\frac{1}{a_1}\cdots\frac{1}{a_{2p-1}}+O(N^{-1})\Bigr)\\
& \ \ \times \prod_{j=1}^{2p-1}(-\xi_j)^{-\frac{1}{2}}
\cdot \exp \Bigl(\frac{\sqrt{-1}N}{2\pi}({\rm Re}[f(a_1,\ldots,a_{2p-1})]-\frac{\pi^2 (2p-2)}{6})\Bigr).
\end{align*}
As for symbol $\xi_j$, see the Lemma \ref{122}.
We know the following equality:
$$
{\rm Li}_2(\exp (2\pi \sqrt{-1} z))+{\rm Li}_2(\exp (-2\pi \sqrt{-1}z))
=-4\pi ^2 \zeta (-1,z) \ \ ({\rm Im}[z]\geq 0).
$$
Here, $\zeta (a,z)$ is the Hurwitz zeta function. Hence, domain of ${\rm Li}_2(z)$ is $\C$.

${\rm Im}[f(a_1,\ldots,a_{2p-1})]$ is the hyperbolic volume of a twist knot \cite{1}.

%Let the path from $\lambda_i$ to $\mu_i$ be $\gamma_i$.
%The $\gamma_i$ are Brachistochrone curve.
%Because of $(a_1,\ldots ,a_{2p-1})\in \C^{2p-1}$ is saddle point,
%we satisfy the following equality:
%$$
%\max_{z_i\in \gamma_i}f(z_1,\ldots ,z_{2p-1})=f(a_1,\ldots,a_{2p-1}).
%$$
We prove the following Lemma in this section:
\begin{lem}\label{122}
\begin{align*}
&\int_{0}^N\cdots \int_{0}^N
\exp \Bigl(\frac{\sqrt{-1}N}{2\pi}(f(\exp (2\pi\sqrt{-1}n_1/N),\ldots ,\exp (2\pi \sqrt{-1}n_{2p-1}/N))\\
& \ \ -\frac{\pi^2(2p-2)}{6})\Bigr)dn_1\cdots dn_{2p-1}\\
&=
\exp \Bigl(-\frac{3\pi\sqrt{-1} (2p-1)}{4}\Bigr) \cdot N^{\frac{2p-1}{2}}
\cdot \exp \Bigl(-\frac{N}{2\pi}{\rm Im}[f(a_1,\ldots ,a_{2p-1})]\Bigr)\\
& \ \ \times \Bigl(\frac{1}{a_1}\cdots\frac{1}{a_{2p-1}}+O(N^{-1})\Bigr)\\
& \ \ \times \prod_{j=1}^{2p-1}(-\xi_j)^{-\frac{1}{2}}
\cdot \exp \Bigl(\frac{\sqrt{-1}N}{2\pi}({\rm Re}[f(a_1,\ldots,a_{2p-1})]-\frac{\pi^2 (2p-2)}{6})\Bigr),
\end{align*}
where $\xi_j$ are eigenvalues of the Hessian matrix of $f(z_1,\ldots ,z_{2p-1})$
evaluated at $(a_1,\ldots,a_{2p-1})$ and $(-\xi_j)^{-\frac{1}{2}}$ are defined with arguments:
$$
|{\rm Arg}\sqrt{-\xi_j}|<\pi.
$$
\end{lem}
\begin{proof}
We regard $\exp(2\pi\sqrt{-1} n_i/N)$ as $z_i$. We obtain the following equality:
$$
\exp (\frac{2\pi\sqrt{-1} n_i}{N})\frac{2\pi\sqrt{-1}}{N}dn_i=dz_i.
$$
Hence,
\begin{align*}
&\int_{0}^N
\cdots \int_{0}^N
\exp \Bigl(\frac{\sqrt{-1}N}{2\pi}(f(\exp(2\pi \sqrt{-1}n_1/N),\ldots ,\exp(2\pi \sqrt{-1}n_{2p-1}/N))\\
& \ \ -\frac{\pi^2(2p-2)}{6})\Bigr)dn_1\cdots dn_{2p-1}\\
&=\int_{|z_{2p-1}|=1}\cdots \int_{|z_1|=1}\Bigl(\frac{N}{2\pi \sqrt{-1}}\Bigr)^{2p-1}
\exp \Bigl(\frac{\sqrt{-1}N}{2\pi}(f(z_1,\ldots ,z_{2p-1})-\frac{\pi^2(2p-2)}{6})\Bigr)\\
& \ \ \times\frac{1}{z_1}\cdots \frac{1}{z_{2p-1}}
dz_1\cdots dz_{2p-1}.
\end{align*}
We deform $|z_i|=1$ into $C_i$ so that $a_i$ will be a maximal value ${\rm Im}[f(a_1,\ldots,a_{2p-1})]$.
%As you know, $\gamma_i \subset C_i$.
Because of the method of steepest descent, we obtain the following formulas:
\begin{align*}
&\int_{|z_{2p-1}|=1}\cdots \int_{|z_1|=1}\Bigl(\frac{N}{2\pi \sqrt{-1}}\Bigr)^{2p-1}
\exp \Bigl(\frac{\sqrt{-1}N}{2\pi}(f(z_1,\ldots ,z_{2p-1})
- \frac{\pi^2 (2p-2)}{6})\Bigr)\\
& \ \ \times \frac{1}{z_1}\cdots \frac{1}{z_{2p-1}}dz_1\cdots dz_{2p-1}\\
&=\int_{C_{2p-1}}\cdots \int_{C_1}\Bigl(\frac{N}{2\pi \sqrt{-1}}\Bigr)^{2p-1}
\exp \Bigl(\frac{\sqrt{-1}N}{2\pi}(f(z_1,\ldots ,z_{2p-1})-\frac{\pi^2 (2p-2)}{6})\Bigr)\\
& \ \ \times \frac{1}{z_1}\cdots \frac{1}{z_{2p-1}}dz_1\cdots dz_{2p-1}\\
&= \Bigl(\frac{N}{2\pi \sqrt{-1}}\Bigr)^{2p-1}
\Bigl(\frac{4\pi^2}{\sqrt{-1}N}\Bigr)^{\frac{2p-1}{2}}
\exp \Bigl(\frac{\sqrt{-1}N}{2\pi}(f(a_1,\ldots,a_{2p-1})-\frac{\pi^2(2p-2)}{6})\Bigr)\\
& \ \ \times \Bigl(\frac{1}{a_1}\cdots\frac{1}{a_{2p-1}}+O(N^{-1})\Bigr)
\prod_{j=1}^{2p-1}(-\xi_j)^{-\frac{1}{2}}\\
&=\exp \Bigl(-\frac{3\pi \sqrt{-1} (2p-1)}{4}\Bigr) \cdot N^{\frac{2p-1}{2}}
\cdot \exp \Bigl(-\frac{N}{2\pi}{\rm Im}[f(a_1,\ldots ,a_{2p-1})]\Bigr)\\
& \ \ \times \Bigl(\frac{1}{a_1}\cdots\frac{1}{a_{2p-1}}+O(N^{-1})\Bigr)\\
& \ \ \times \prod_{j=1}^{2p-1}(-\xi_j)^{-\frac{1}{2}}
\cdot \exp \Bigl(\frac{\sqrt{-1}N}{2\pi}({\rm Re}[f(a_1,\ldots,a_{2p-1})]-\frac{\pi^2 (2p-2)}{6})\Bigr).
%&=\Bigl(N^{\frac{2p-1}{2}}\exp (\frac{N}{2\pi}{\rm Im}[f(a_1,\ldots ,a_{2p-1})])\Bigr)\\
%& \ \ \times\Bigl|{\rm Re}[\frac{1}{a_1}\cdots\frac{1}{a_{2p-1}}\prod_{j=1}^{2p-1}(-\xi_j)^{-\frac{1}{2}}]
%\cos (\frac{\pi(2p-1)}{4}-\frac{-N}{2\pi}{\rm Re}[f(a_1,\ldots,a_{2p-1})])\cos (\frac{\pi^2 (2p-2)}{6})\Bigr|\\
%& \ \ +\Bigl|{\rm Re}[\frac{1}{a_1}\cdots\frac{1}{a_{2p-1}}\prod_{j=1}^{2p-1}(-\xi_j)^{-\frac{1}{2}}]
%\sin (\frac{\pi(2p-1)}{4}-\frac{-N}{2\pi}{\rm Re}[f(a_1,\ldots,a_{2p-1})])\cos (\frac{\pi^2 (2p-2)}{6})\Bigr|\\
%& \ \ +\Bigl|{\rm Im}[\frac{1}{a_1}\cdots\frac{1}{a_{2p-1}}\prod_{j=1}^{2p-1}(-\xi_j)^{-\frac{1}{2}}]
%\sin (\frac{\pi(2p-1)}{4}-\frac{-N}{2\pi}{\rm Re}[f(a_1,\ldots,a_{2p-1})])\sin (\frac{\pi^2 (2p-2)}{6})\Bigr|\\
%& \ \ +\Bigl|{\rm Im}[\frac{1}{a_1}\cdots\frac{1}{a_{2p-1}}\prod_{j=1}^{2p-1}(-\xi_j)^{-\frac{1}{2}}]
%\cos (\frac{\pi(2p-1)}{4}-\frac{-N}{2\pi}{\rm Re}[f(a_1,\ldots,a_{2p-1})])\sin (\frac{\pi^2 (2p-2)}{6})\Bigr|\\
%&\leq \Bigl(2\sqrt{2}N^{\frac{2p-1}{2}}\exp (\frac{N}{2\pi}{\rm Im}[f(a_1,\ldots ,a_{2p-1})])\Bigr)\\
%& \ \ \times\Bigl(
%\Bigl|{\rm Re}[\frac{1}{a_1}\cdots\frac{1}{a_{2p-1}}\prod_{j=1}^{2p-1}(-\xi_j)^{-\frac{1}{2}}]\Bigr|
%+\Bigl|{\rm Im}[\frac{1}{a_1}\cdots\frac{1}{a_{2p-1}}\prod_{j=1}^{2p-1}(-\xi_j)^{-\frac{1}{2}}]\Bigr|\Bigr),
\end{align*}
%and
%\begin{align*}
%&\int_{C_{2p-1}}\cdots \int_{C_1}\Bigl(\frac{N}{2\pi \sqrt{-1}}\Bigr)^{2p-1}
%\exp (\frac{-\sqrt{-1}N}{2\pi}f(z_1,\ldots ,z_{2p-1}))\exp (\frac{\pi^2\sqrt{-1} (2p-2)}{6})\\
%& \ \ \times\frac{1}{z_1}\cdots \frac{1}{z_{2p-1}}dz_1\cdots dz_{2p-1}\\
%&\overset{N}{\gtrsim} \Bigl|\Bigl(\frac{N}{2\pi \sqrt{-1}}\Bigr)^{2p-1}
%\Bigl(\frac{4\pi^2}{-\sqrt{-1}N}\Bigr)^{\frac{2p-1}{2}}
%\exp (\frac{-\sqrt{-1}N}{2\pi}f(a_1,\ldots,a_{2p-1}))\exp (\frac{\pi^2\sqrt{-1}(2p-2)}{6})\\
%& \ \ \times \frac{1}{a_1}\cdots\frac{1}{a_{2p-1}}\prod_{j=1}^{2p-1}(-\xi_j)^{-\frac{1}{2}}\Bigr|\\
%&=\Bigl(N^{\frac{2p-1}{2}}\exp (\frac{N}{2\pi}{\rm Im}[f(a_1,\ldots ,a_{2p-1})])\Bigr)
%\Bigl| \frac{1}{a_1}\cdots\frac{1}{a_{2p-1}}\prod_{j=1}^{2p-1}(-\xi_j)^{-\frac{1}{2}}\Bigr|.
%\end{align*}
Therefore, we obtain the claim.
\end{proof}

\begin{cor}\label{123}With the same notation as in Lemma \ref{122} we have that
\begin{align*}
&\int_{0}^N\cdots \int_{0}^N
\exp \Bigl(\frac{-\sqrt{-1}N}{2\pi}\\
& \ \ \times f(\exp (2\pi\sqrt{-1}n_1/N),\ldots ,\exp (2\pi \sqrt{-1}n_{2p-1}/N))
\Bigr) dn_1\cdots dn_{2p-1}\\
&=
\exp \Bigl(-\frac{\pi \sqrt{-1} (2p-1)}{4}\Bigr) \cdot N^{\frac{2p-1}{2}}
\cdot \exp \Bigl(\frac{N}{2\pi}{\rm Im}[f(a_1,\ldots ,a_{2p-1})]\Bigr)\\
& \ \ \times \Bigl(\frac{1}{a_1}\cdots\frac{1}{a_{2p-1}}+O(N^{-1})\Bigr)\\
& \ \ \times \prod_{j=1}^{2p-1}(-\xi_j)^{-\frac{1}{2}}
\cdot \exp \Bigl(\frac{-\sqrt{-1}N}{2\pi}{\rm Re}[f(a_1,\ldots,a_{2p-1})]\Bigr).
\end{align*}
\end{cor}

\section{Proof of the volume conjecture for twist knots}
In this section, we prove Volume Conjecture for twist knots by using Lemmas of the previous Section.

\begin{lem}\label{106}
With the same notation as in Theorem \ref{155}. If
$$
\frac{2\pi\log |\psi(N)|}{N}\to 0 \ \ (N\to \infty),
$$
then we have that
\begin{align*}
&\frac{2 \pi}{N}\log \Bigl(|\psi(N)|
|(q)_{n_{2p-1}}|\prod_{i=1}^{2p-2} \Bigl|{\begin{bmatrix}n_{i+1}\\n_i\end{bmatrix}}\Bigr|\Bigr)\\
&\overset{N}{\sim} {\rm Im}[f({\rm exp}(2 \pi \sqrt{-1} n_1/N),\ldots ,{\rm exp}(2 \pi \sqrt{-1} n_{2p-1}/N))]\\
&\overset{N}{\sim} -\sqrt{-1}f({\rm exp}(2 \pi \sqrt{-1} n_1/N),\ldots ,{\rm exp}(2 \pi \sqrt{-1} n_{2p-1}/N))
+\frac{\pi^2\sqrt{-1}(2p-2)}{6}\\
& \ \ +\sum_{i=1}^{2p-2}(-g(n_i,N)+g(n_{i+1},N)-g(n_{i+1}-n_i,N)).
\end{align*}
\end{lem}
\begin{proof}We obtain the following formulas:
\begin{align}\nonumber
&\frac{2 \pi}{N}\log \Bigl(|\psi(N)||(q)_{n_{2p-1}}|\prod_{i=1}^{2p-2} \Bigl|{\begin{bmatrix}n_{i+1}\\n_i\end{bmatrix}}
\Bigr|\Bigr)\\ \nonumber
&= \frac{2 \pi}{N}\Bigl(\log |\psi(N)|+\log |(q)_{n_{2p-1}}|+\sum_{i=1}^{2p-2} \log \Bigl|{\begin{bmatrix}n_{i+1}\\n_i\end{bmatrix}}\Bigr|\Bigr)\\ \nonumber
&\overset{N}{\sim} \frac{2 \pi}{N}\Bigl(\log |(q)_{n_{2p-1}}|+\sum_{i=1}^{2p-2} \log \Bigl|{\begin{bmatrix}n_{i+1}\\n_i\end{bmatrix}}\Bigr|\Bigr)\\ \nonumber
&\overset{N}{\sim} \sum_{i=1}^{2p-2} {\rm Im}[{\rm Li}_2({\rm exp}(2 \pi \sqrt{-1} n_i/N))-{\rm Li}_2({\rm exp}(2 \pi \sqrt{-1} n_{i+1}/N))\\ \nonumber
& \ \ +{\rm Li}_2({\rm exp}(2 \pi \sqrt{-1} (n_{i+1}-n_i)/N))]\\ \label{107}
& \ \ +{\rm Im}[- {\rm Li}_2({\rm exp}(2 \pi \sqrt{-1} n_{2p-1}/N))]\\ \nonumber
&={\rm Im}[f(z^{(1)}_1,\ldots ,z^{(2p-1)}_{2p-1})]\\ \label{108}
&\overset{N}{\sim} {\rm Im}[f({\rm exp}(2 \pi \sqrt{-1} n_1/N),\ldots ,{\rm exp}(2 \pi \sqrt{-1} n_{2p-1}/N))].
\end{align}
Here, Equivalent relation (\ref{107}) is shown by using Lemma \ref{101} and Corollary \ref{102}.
Equivalent relation (\ref{108}) is shown by using Lemma \ref{103}. Moreover,
\begin{align}\nonumber
&\frac{2 \pi}{N}\log \Bigl(|\psi(N)||(q)_{n_{2p-1}}|\prod_{i=1}^{2p-2} \Bigl|{\begin{bmatrix}n_{i+1}\\n_i\end{bmatrix}}
\Bigr|\Bigr)\\ \nonumber
&= \frac{2 \pi}{N}\Bigl(\log |\psi(N)|+\log |(q)_{n_{2p-1}}|+\sum_{i=1}^{2p-2} \log \Bigl|{\begin{bmatrix}n_{i+1}\\n_i\end{bmatrix}}\Bigr|\Bigr)\\ \nonumber
&\overset{N}{\sim} \frac{2 \pi}{N}\Bigl(\log |(q)_{n_{2p-1}}|+\sum_{i=1}^{2p-2} \log \Bigl|{\begin{bmatrix}n_{i+1}\\n_i\end{bmatrix}}\Bigr|\Bigr)\\ \nonumber
&\overset{N}{\sim} \tilde{f}(z^{(1)}_1,\ldots ,z^{(2p-1)})\\ \label{110}
& \ \ +\sum_{i=1}^{2p-2}(-g(n_i,N)+g(n_{i+1},N)-g(n_{i+1}-n_i,N))\\ \nonumber
&\overset{N}{\sim}-\sqrt{-1}f({\rm exp}(2 \pi \sqrt{-1} n_1/N),\ldots ,{\rm exp}(2 \pi \sqrt{-1} n_{2p-1}/N))
\\ \nonumber
& \ \ +\frac{\pi^2\sqrt{-1}(2p-2)}{6}\\ \label{111}
& \ \ +\sum_{i=1}^{2p-2}(-g(n_i,N)+g(n_{i+1},N)-g(n_{i+1}-n_i,N)).
\end{align}
Here, Equivalent relation (\ref{110}) is shown by using Lemma \ref{101} and Corollary \ref{102}.
Equivalent relation (\ref{111}) is shown by using Lemma \ref{103}.
\end{proof}

\begin{lem}\label{109}
With the same notation as in Theorem \ref{155} we have that
$$
\lim_{N\to \infty}\frac{2 \pi \log |J_{N}(K^{*}_{p>0};q)|}{N} \leq {\rm Im}[f(a_1,\ldots,a_{2p-1})].
$$
\end{lem}
\begin{proof}
We obtain the following inequality:
$$
\frac{2 \pi \log |J_N(K^{*}_{p>0};q)|}{N} \leq \frac{2 \pi}{N} \log \sum_{N-1\geq m_{2p-1}\geq \cdots \geq m_1 \geq 0} |(q)_{m_{2p-1}}|\prod_{i=1}^{2p-2} \Bigl|{\begin{bmatrix}m_{i+1}\\m_i\end{bmatrix}}
\Bigr|.
$$
When $N-1\geq m_{2p-1}\geq \cdots \geq m_1 \geq 0$, because of \cite{12},
$$
\Gamma_q(z+1)=\frac{1-q^z}{1-q}\Gamma_q(z),
$$
and
$$
{\begin{bmatrix}m_{{2p-1}}+1\\ s\end{bmatrix}}={\begin{bmatrix}m_{{2p-1}}\\ s-1 \end{bmatrix}}+
q^s{\begin{bmatrix}m_{{2p-1}}\\ s\end{bmatrix}}
$$
hold.
By using the above equalities, 
we obtain the following formulas:\\
(i) \ \ In case of $m_1$,
$$
\Biggl|\frac{{\begin{bmatrix}m_{2}\\ m_1+1 \end{bmatrix}}}{\begin{bmatrix}m_{2}\\ m_1 \end{bmatrix}}
\Biggr|=\Bigl|\frac{(q)_{m_2-m_1}(q)_{m_1}}{(q)_{m_2-(m_1+1)}(q)_{m_1+1}}\Bigr|=|(1-q)^{m_2-2m_1-1}|
\lesseqgtr 1.
$$
Hence,
$$
m_1=\frac{m_2 \log (1-q) -\log (1 -q) -\sqrt{-1}\pi}{2\log (1-q)}\to \frac{m_2-1}{2} \ \ (N\to \infty).
$$
In fact, when we substitute $m_1=(m_2-1)/2 $ for $|(1-q)^{m_2-2m_1-1}|$, we obtain the following equality:
$$
|(1-q)^{m_2-(m_2-1)-1}|=1.
$$
(ii) \ \ In case of $m_2, \ldots ,m_{2p-2}$,
$$
\Biggl|\frac{{\begin{bmatrix}m_{i+1} \\ m_{i}+1 \end{bmatrix}}{\begin{bmatrix}m_{i}+1 \\
m_{i-1} \end{bmatrix}}}{{\begin{bmatrix}m_{i+1} \\ m_{i} \end{bmatrix}}{\begin{bmatrix}m_{i} \\
m_{i-1} \end{bmatrix}}}\Biggr|
=\Bigl|\frac{(q)_{m_{i+1}-m_{i}} (q)_{m_{i}-m_{i-1}}}{(q)_{m_{i+1}-m_{i}-1} (q)_{m_{i}-m_{i-1}+1}}\Bigr|
=|(1-q)^{m_{i+1}-2m_{i}+m_{i-1}-1}|\lesseqgtr 1.
$$
Hence,
\begin{align*}
m_{i}&=\frac{2m_{i+1} \log 2+2m_{i-1} \log 2+m_{i+1} \log (\sin ^2 \frac{\pi}{N})
+m_{i-1}\log(\sin ^2 \frac{\pi}{N})
-2\log (2 \sqrt{\sin ^2 \frac{\pi}{N}})}{4\log 2+ 2\log (\sin ^2\frac{\pi}{N})}\\
&\to \frac{m_{i+1}+m_{i-1}-1}{2} \ \ (N\to \infty).
\end{align*}
In fact, when we substitute $m_i=(m_{i+1}+m_{i-1}-1)/2$ for $|(1-q)^{m_{i+1}-2m_{i}+m_{i-1}-1}|$,
we obtain the following equality:
$$
|(1-q)^{m_{i+1}-(m_{i+1}+m_{i-1}-1)+m_{i-1}-1}|=1.
$$
(iii) \ \ In case of $m_{2p-1}$,
$$
\Biggl| \frac{(q)_{m_{2p-1}+1}}{(q)_{m_{2p-1}}}\Biggr|=|(1-q)^{m_{2p-1}+1}|\lesseqgtr 1.
$$
Hence, 
$$
m_{2p-1}=\frac{\pi\sqrt{-1} -\log (1-q)}{\log (1 - q)}\to -1 \ \ (N\to \infty)
$$
In fact, when we substitute $m_{2p-1}=-1$ for $|(1-q)^{m_{2p-1}+1}|$, we obtain the following equality:
$$
|(1-q)^{-1+1}|=1.
$$
Thus,
$$
\cdots \leq |(q)_{-2}|\leq |(q)_{-1}|\geq |(q)_0|\geq |(q)_1|\geq \cdots.
$$
Consequently, $m_{2p-1}=0$ is a maximal value of $|(q)_0|=1$.
Moreover,
$$
\Biggl|\frac{{\begin{bmatrix}m_{2p-1}+1\\ m_{2p-2}\end{bmatrix}}}{\begin{bmatrix}m_{2p-1}\\ 
m_{2p-2} \end{bmatrix}}\Biggr|
=\Bigl|q^{m_{2p-2}}+\frac{{\begin{bmatrix}m_{2p-1}\\ m_{2p-2}-1\end{bmatrix}}}{\begin{bmatrix}m_{2p-1}\\
m_{2p-2} \end{bmatrix}}\Bigr|=\Bigl|q^{m_{2p-2}}+\frac{1-q^{m_{2p-2}}}{1-q^{m_{2p-1}-m_{2p-2}+1}}\Bigr|
\lesseqgtr 1.
$$
Hence,
$$
m_{2p-1}=\frac{2(\pi m_{2p-2}-\pi)+\sqrt{-1}N \log \frac{1+q^{m_{2p-2}}}{2}}{2\pi} \to \frac{m_{2p-2}}{2}-1 \ \ (N\to \infty).
$$
In fact, when we substitute $m_{2p-1}=m_{2p-2}/2 -1$ for 
$|q^{m_{2p-2}}+(1-q^{m_{2p-2}})/(1-q^{m_{2p-1}-m_{2p-2}+1})|$,
we obtain the following equality:
$$
\Bigl|q^{m_{2p-2}}+\frac{1-q^{m_{2p-2}}}{1-q^{\frac{m_{2p-2}}{2}-1-m_{2p-2}+1}}\Bigr|=
\Bigl|\frac{1-q^{\frac{m_{2p-2}}{2}}}{1-q^{-\frac{m_{2p-2}}{2}}}\Bigr|=1.
$$
Therefore, there exist the following positive real sequence $n_i\in \N \ \ (1\leq i\leq 2p-1)$:
\begin{equation*}
n_2=2n_1+1, \ \ n_{i+1}=2n_i-n_{i-1}+1 \ \ (i=2,\ldots ,2p-3), \ \ \text{and} \ \ 
n_{2p-1}=\frac{n_{2p-2}}{2}-1
\end{equation*}
such that for every $(m_1,\ldots, m_{2p-1})\in \N^{2p-1}$,
$$
|(q)_{m_{2p-1}}|\prod_{i=1}^{2p-2} \Bigl|{\begin{bmatrix}m_{i+1} \\m_{i} \end{bmatrix}}\Bigr|
\leq |(q)_{0}|\prod_{i=1}^{2p-2} \Bigl|{\begin{bmatrix}n_{i+1} \\n_{i} \end{bmatrix}}\Bigr|.
$$
Hence,
\begin{align}\nonumber
\frac{2 \pi \log |J_N(K^{*}_{p>0};q)|}{N}& \leq \frac{2 \pi}{N} \log \sum_{N-1\geq m_{2p-1}\geq \cdots \geq m_1 \geq 0} |(q)_{m_{2p-1}}|\prod_{i=1}^{2p-2} \Bigl|{\begin{bmatrix}m_{i+1}\\m_i\end{bmatrix}}
\Bigr|\\ \nonumber
&\leq \frac{2 \pi}{N} \log \sum_{N-1\geq m_{2p-1}\geq \cdots \geq m_1 \geq 0} |(q)_{0}|\prod_{i=1}^{2p-2} \Bigl|{\begin{bmatrix}n_{i+1}\\n_i\end{bmatrix}}
\Bigr|\\ \nonumber
&\leq \frac{2 \pi}{N} \log \Bigl( N^{2p-1} |(q)_{0}| \prod_{i=1}^{2p-2}\Bigl|{\begin{bmatrix}n_{i+1}\\n_i\end{bmatrix}}\Bigr| \Bigr)\\ \nonumber
&= \frac{2 \pi}{N}\Bigl( \log N^{2p-1}+\log |(q)_{0}|
+\sum_{i=1}^{2p-2} \log \Bigl|{\begin{bmatrix}n_{i+1}\\n_i\end{bmatrix}}\Bigr|\Bigr)\\ \nonumber
&\overset{N}{\sim} \tilde{f}(z_1^{(1)},\ldots,z_{2p-1}^{(2p-1)})
-\sqrt{-1}{\rm Li}_2({\rm exp}(2 \pi \sqrt{-1} n_{2p-1}/N))\\ \nonumber
& \ \ +\frac{\pi^2\sqrt{-1}}{6}
+\sum_{i=1}^{2p-2}(-g(n_i,N)+g(n_{i+1},N)-g(n_{i+1}-n_i,N))\\ \nonumber
&=-\sqrt{-1}(f(z_1^{(1)},\ldots,z_{2p-1}^{(2p-1)})
+{\rm Li}_2(z_{2p-1}^{(2p-1)})-\frac{\pi^2}{6})\\ \nonumber
& \ \ +\frac{\pi^2\sqrt{-1}(2p-2)}{6}\\ \nonumber
& \ \ -g(n_1,N)+g(n_{2p-1},N)-\sum_{i=1}^{2p-2}g(n_{i+1}-n_i,N) \\ \nonumber
&\overset{N}{\sim} 
-\sqrt{-1}f({\rm exp}(2 \pi \sqrt{-1} n_1/N),\ldots ,{\rm exp}(2 \pi \sqrt{-1} n_{2p-1}/N))\\ \nonumber
& \ \ +\frac{\pi^2 \sqrt{-1}(2p-2)}{6}\\ \label{112}
& \ \ -g(n_1,N)+g(n_{2p-1},N)-\sum_{i=1}^{2p-2}g(n_{i+1}-n_i,N).
\end{align}
Here,  
Equivalent relation (\ref{112}) is shown by using Lemma \ref{103}.
Because of Lemma \ref{106}, we obtain the following formulas:
\begin{align*}
-\log |J_N(K^{*}_{p>0};q)|& \geq 
\frac{N}{2\pi}\Bigl(\sqrt{-1}f({\rm exp}(2 \pi \sqrt{-1} n_1/N),\ldots ,{\rm exp}(2 \pi \sqrt{-1} n_{2p-1}/N))\\
& \ \ -\frac{\pi^2 \sqrt{-1}(2p-2)}{6}\Bigr)\\
& \ \ -\frac{N}{2\pi}\Bigl( -g(n_1,N)+g(n_{2p-1},N)-\sum_{i=1}^{2p-2}g(n_{i+1}-n_i,N)\Bigr),
\end{align*}
\begin{align*}
\frac{1}{|J_N(K^{*}_{p>0};q)|}&\geq \exp \Bigl(\frac{\sqrt{-1}N}{2\pi}(f(\exp (2\pi \sqrt{-1}n_1/N),\ldots ,
\exp (2\pi \sqrt{-1} n_{2p-1}/N))\\
& \ \ -\frac{\pi^2(2p-2)}{6})\\
& \ \ -\frac{N}{2\pi}(-g(n_1,N)+g(n_{2p-1},N)-\sum_{i=1}^{2p-2}g(n_{i+1}-n_i,N))\Bigr) \in \R_{>0},
\end{align*}
and
\begin{align*}
\frac{1}{|J_N(K^{*}_{p>0};q)|}&\geq \Bigl|\exp \Bigl(\frac{\sqrt{-1}N}{2\pi}(f(\exp (2\pi \sqrt{-1}n_1/N),\ldots ,
\exp (2\pi \sqrt{-1} n_{2p-1}/N))\\
& \ \ -\frac{\pi^2(2p-2)}{6})\\
& \ \ -\frac{N}{2\pi}(-g(n_1,N)+g(n_{2p-1},N)-\sum_{i=1}^{2p-2}g(n_{i+1}-n_i,N))\Bigr)\Bigr|.
\end{align*}
Since for every positive real number $n\in \R_{>0}$,
\begin{align*}
&\Bigl|\exp \Bigl(\pm \frac{N}{2\pi}g(n,N)\Bigr)\Bigr|\\
&=
\Bigl|\exp \Bigl(\pm \Bigl(\frac{\pi \sqrt{-1} n^2}{2N}+ n \log 2-n \log (1 - \exp (2\pi \sqrt{-1}n/N))+
\frac{n}{2}\log \sin ^2 \Bigl(\frac{\pi n}{N}\Bigr)\Bigr)\Bigr)\Bigr|\\
&=1
\end{align*}
holds, we obtain the following equations:
\begin{align*}
& \Bigl|\exp \Bigl(\frac{\sqrt{-1}N}{2\pi}(f(\exp (2\pi \sqrt{-1}n_1/N),\ldots ,
\exp (2\pi \sqrt{-1} n_{2p-1}/N))
 -\frac{\pi^2(2p-2)}{6})\\
& \ \ -\frac{N}{2\pi}(-g(n_1,N)+g(n_{2p-1},N)-\sum_{i=1}^{2p-2}g(n_{i+1}-n_i,N))\Bigr)\Bigr|\\
&=\Bigl|\exp \Bigl(\frac{\sqrt{-1}N}{2\pi}(f(\exp (2\pi \sqrt{-1}n_1/N),\ldots ,
\exp (2\pi \sqrt{-1} n_{2p-1}/N))-\frac{\pi^2(2p-2)}{6})\Bigr)\Bigr|.
\end{align*}
Moreover,
\begin{align*}
&\int_{n_{2p-1}}^{N+n_{2p-1}}\int_{n_{2p-2}}^{2N+n_{2p-2}}\cdots
\int_{n_2}^{2N+n_2}\int_{n_1}^{N+n_1}\frac{1}{|J_N(K^{*}_{p>0};q)|}dn_1dn_2\cdots
dn_{2p-2}d_{2p-1}\\
&=2^{2p-3}\int_{0}^{N}\int_{0}^{N}\cdots
\int_{0}^{N}\int_{0}^{N}\frac{1}{|J_N(K^{*}_{p>0};q)|}dn_1dn_2\cdots
dn_{2p-2}d_{2p-1},
\end{align*}
and
\begin{align*}
&\int_{n_{2p-1}}^{N+n_{2p-1}}\int_{n_{2p-2}}^{2N+n_{2p-2}}\cdots
\int_{n_2}^{2N+n_2}\int_{n_1}^{N+n_1}\\
& \ \ \Bigl|\exp \Bigl(\frac{\sqrt{-1}N}{2\pi}(f(\exp (2\pi \sqrt{-1}n_1/N),\ldots ,
\exp (2\pi \sqrt{-1} n_{2p-1}/N))-\frac{\pi^2(2p-2)}{6})\Bigr)\Bigr|\\
& \ \ dn_1dn_2\cdots dn_{2p-2}d_{2p-1}\\
&=\int_{0}^{N}\int_{0}^{2N}\cdots
\int_{0}^{2N}\int_{0}^{N}\\
& \ \ \Bigl|\exp \Bigl(\frac{\sqrt{-1}N}{2\pi}(f(\exp (2\pi \sqrt{-1}n_1/N),\ldots ,
\exp (2\pi \sqrt{-1} n_{2p-1}/N))
-\frac{\pi^2(2p-2)}{6})\Bigr)\Bigr|\\
&dn_1dn_2\cdots dn_{2p-2}d_{2p-1}\\
&=2^{2p-3}\int_{|z_{2p-1}|=1}\int_{|z_{2p-2}|=1}\cdots \int_{|z_2|=1}\int_{|z_1|=1}
\Bigl(\frac{N}{2\pi \sqrt{-1}}\Bigr)^{2p-1}\\
& \ \ \times \exp \Bigl(\frac{\sqrt{-1}N}{2\pi}(f(z_1,\ldots ,z_{2p-1})
- \frac{\pi^2 (2p-2)}{6})\Bigr)
 \frac{1}{z_1}\frac{1}{z_2}\cdots \frac{1}{z_{2p-2}}\frac{1}{z_{2p-1}}\\
& \ \ dz_1 dz_2\cdots dz_{2p-2}dz_{2p-1}\\
&=2^{2p-3}\int_0^N\int_0^N\cdots \int_0^N\int_0^N\\
& \ \ \Bigl|\exp \Bigl(\frac{\sqrt{-1}N}{2\pi}(f(\exp (2\pi \sqrt{-1}n_1/N),\ldots ,
\exp (2\pi \sqrt{-1} n_{2p-1}/N))
-\frac{\pi^2(2p-2)}{6})\Bigr)\Bigr|\\
&dn_1dn_2\cdots dn_{2p-2}d_{2p-1}.
\end{align*}
Hence,
\begin{align*}
&\int_{0}^{N}\cdots \int_{0}^{N}
\frac{1}{|J_N(K^{*}_{p>0};q)|} dn_1\cdots dn_{2p-1}\\
%&\geq \int_0^N\cdots \int_0^N
%\Bigl|\exp \Bigl(\frac{\sqrt{-1}N}{2\pi}(f(\exp (2\pi \sqrt{-1}n_1/N),\ldots ,
%\exp (2\pi \sqrt{-1} n_{2p-1}/N))\\
%& \ \ -\frac{\pi^2(2p-2)}{6})\\
%& \ \ -\frac{N}{2\pi}(-g(n_1,N)+g(n_{2p-1},N)-\sum_{i=1}^{2p-2}g(n_{i+1}-n_i,N))\Bigr)
%\Bigr|dn_1\cdots dn_{2p-2}\\
&\geq \int_0^N\cdots \int_0^N
\Bigl|\exp \Bigl(\frac{\sqrt{-1}N}{2\pi}(f(\exp (2\pi \sqrt{-1}n_1/N),\ldots ,
\exp (2\pi \sqrt{-1} n_{2p-1}/N))\\
& \ \ -\frac{\pi^2(2p-2)}{6})\Bigr)
\Bigr|dn_1\cdots dn_{2p-2}\\
&\geq \Bigl|\int_0^N\cdots \int_0^N
\exp \Bigl(\frac{\sqrt{-1}N}{2\pi}(f(\exp (2\pi \sqrt{-1}n_1/N),\ldots ,
\exp (2\pi \sqrt{-1} n_{2p-1}/N))\\
& \ \ -\frac{\pi^2(2p-2)}{6})\Bigr)
dn_1\cdots dn_{2p-2}\Bigr|.
\end{align*}
Because of Lemma \ref{122}, we obtain the following formulas:
\begin{align*}
&\frac{N^{2p-1}}{|J_N(K^{*}_{p>0};q)|}\\
&\geq 
\Bigl|\exp \Bigl(-\frac{3\pi\sqrt{-1} (2p-1)}{4}\Bigr) \cdot N^{\frac{2p-1}{2}}
\cdot \exp \Bigl(-\frac{N}{2\pi}{\rm Im}[f(a_1,\ldots ,a_{2p-1})]\Bigr)\\
& \ \ \times \Bigl(\frac{1}{a_1}\cdots\frac{1}{a_{2p-1}}+O(N^{-1})\Bigr)\\
& \ \ \times \prod_{j=1}^{2p-1}(-\xi_j)^{-\frac{1}{2}}
\cdot \exp \Bigl(\frac{\sqrt{-1}N}{2\pi}({\rm Re}[f(a_1,\ldots,a_{2p-1})]-\frac{\pi^2 (2p-2)}{6})\Bigr)\Bigr|\\
&=N^{\frac{2p-1}{2}}
\cdot \exp \Bigl(-\frac{N}{2\pi}{\rm Im}[f(a_1,\ldots ,a_{2p-1})]\Bigr)\\
& \ \ \times \Bigl|\Bigl(\frac{1}{a_1}\cdots\frac{1}{a_{2p-1}}+O(N^{-1})\Bigr)
\cdot \prod_{j=1}^{2p-1}(-\xi_j)^{-\frac{1}{2}}\Bigr|\\
&=N^{\frac{2p-1}{2}}
\cdot \exp \Bigl(-\frac{N}{2\pi}{\rm Im}[f(a_1,\ldots ,a_{2p-1})]\Bigr)
\cdot T(N).
\end{align*}
Here, 
$$
T(N):=\Bigl|\Bigl(\frac{1}{a_1}\cdots\frac{1}{a_{2p-1}}+O(N^{-1})\Bigr)
\cdot \prod_{j=1}^{2p-1}(-\xi_j)^{-\frac{1}{2}}\Bigr|.
$$
Hence,
$$
\log \frac{N^{\frac{2p-1}{2}}}{|J_N(K^{*}_{p>0};q)|\cdot T(N)}\geq 
-\frac{N}{2\pi}{\rm Im}[f(a_1,\ldots ,a_{2p-1})],
$$
and
$$
\frac{2\pi}{N}\Bigl(-\log N^{\frac{2p-1}{2}}
%+\sum_{i=1}^{2p-1}\log |\gamma_i|
%+\sum_{i=1}^{2p-1}\log\frac{1}{n_i}|\log \frac{\mu_i}{\lambda_i}|
+\log |J_N(K^{*}_{p>0};q)| +\log T(N)\Bigr) 
\leq {\rm Im}[f(a_1,\ldots ,a_{2p-1})].
$$
Since
$$
\lim_{N\to \infty}\frac{2\pi}{N}\log N=\lim_{N\to \infty}\frac{2\pi}{N}\log T(N)=0
$$
holds, we obtain the following formula:
$$
\lim_{N\to \infty}\frac{2\pi\log |J_N(K^{*}_{p>0};q)|}{N}\leq {\rm Im}[f(a_1,\ldots ,a_{2p-1})].
$$
\end{proof}

%\begin{rem}
%We explain that the following inequality holds:
%$$
%|J_N(K^{*}_{p>0};q)|\leq \exp (\frac{-\sqrt{-1}N}{2\pi}(f(\exp (2\pi \sqrt{-1}n_1/N),\ldots ,
%\exp (2\pi \sqrt{-1} n_{2p-1}/N))-\frac{\pi^2(2p-2)}{6})).
%$$
%It is because, for a sufficiently large $N$,
%\begin{align*}
%&(-\sqrt{-1})(f(\exp (2\pi \sqrt{-1}n_1/N),\ldots ,
%\exp (2\pi \sqrt{-1} n_{2p-1}/N))-\frac{\pi^2(2p-2)}{6})\\
%&\overset{N}{\sim} {\rm Im}[f(\exp (2\pi \sqrt{-1}n_1/N),\ldots ,
%\exp (2\pi \sqrt{-1} n_{2p-1}/N))]
%\end{align*}
%holds.
%\end{rem}

\begin{lem}\label{114}
Let $\{ a_s \}$ be a complex sequence. For every $0\leq s\leq N-1$ and $0\leq t\leq N-1$,
$$
 2||a_s|-|a_t||\leq |\sum_{u=0}^{N-1}a_u|
$$
holds.
\end{lem}
\begin{proof}
Because of triangle inequality, we obtain the following inequality:
\begin{align*}
|a_s|&=|a_s+\sum_{u=0, u\neq s}^{N-1}a_u-\sum_{u=0, u\neq s}^{N-1}a_u|\leq
|a_s+\sum_{u=0, u\neq s}^{N-1}a_u|+|-\sum_{u=0, u\neq s}^{N-1}a_u|\\
&\leq |\sum_{u=0}^{N-1}a_u|
+\sum_{u=0, u\neq s}^{N-1}|a_u|.
\end{align*}
Hence,
$$
|a_s|-\sum_{u=0, u\neq s}^{N-1}|a_u|\leq |\sum_{u=0}^{N-1}a_u|.
$$
Similarly,
$$
-|a_t|+\sum_{u=0, u\neq t}^{N-1}|a_u|\geq -|\sum_{u=0}^{N-1}a_u|.
$$
Hence,
$$
-|\sum_{u=0}^{N-1}a_u|\leq 2(|a_s|-|a_t|)\leq |\sum_{u=0}^{N-1}a_u|.
$$
Therefore, we obtain the claim.
\end{proof}

\begin{lem}\label{121}
With the same notation as in Theorem \ref{155} we have that
$$
|J_N(K^{*}_{p>0};q)|\geq 2^{2p-1}|(q)_{0}|\prod_{i=1}^{2p-2} \Bigl|{\begin{bmatrix}n_{i+1} \\n_{i} \end{bmatrix}}\Bigr|-2N^{2p-2}.
$$
\end{lem}
\begin{proof}
Since, for every $m_i, \ m_{i+1}\in\N$,
$$
{\begin{bmatrix}m_{i+1}\\m_i\end{bmatrix}}=0 \ \ (m_{i+1}<m_i)
$$
holds, we obtain the following equality:
\begin{align*}
&q^{1-N}\sum_{N-1> m_1 > \cdots > m_{2p-1}> 0}(q^{1-N})_{m_{2p-1}}q^{-N m_{2p-1}}\\
& \ \ \times \prod_{i=1}^{2p-2}(-1)^{m_i} q^{(-1)^{i}N m_i+\binom{m_i}{2}
-m_im_{i+1}}{\begin{bmatrix}m_{i+1}\\m_i\end{bmatrix}}=0.
\end{align*}
Hence,
\begin{align*}
J_N(K^{*}_{p>0};q) 
&=q^{1-N}\sum_{N-1\geq m_{2p-1} \geq \cdots \geq m_1 \geq 0}(q^{1-N})_{m_{2p-1}}q^{-N m_{2p-1}}\\
& \ \ \times \prod_{i=1}^{2p-2}(-1)^{m_i} q^{(-1)^{i}N m_i+\binom{m_i}{2}
-m_im_{i+1}}{\begin{bmatrix}m_{i+1}\\m_i\end{bmatrix}}\\
& \ \ +q^{1-N}\sum_{N-1> m_1 > \cdots > m_{2p-1}> 0}(q^{1-N})_{m_{2p-1}}q^{-N m_{2p-1}}\\
& \ \ \times \prod_{i=1}^{2p-2}(-1)^{m_i} q^{(-1)^{i}N m_i+\binom{m_i}{2}
-m_im_{i+1}}{\begin{bmatrix}m_{i+1}\\m_i\end{bmatrix}}\\
&=q^{1-N}\sum_{m_{2p-1}, \cdots , m_1= 0}^{N-1}
(q^{1-N})_{m_{2p-1}}q^{-N m_{2p-1}}\\
& \ \ \times \prod_{i=1}^{2p-2}(-1)^{m_i} q^{(-1)^{i}N m_i+\binom{m_i}{2}
-m_i m_{i+1}}{\begin{bmatrix}m_{i+1}\\m_i\end{bmatrix}}.
\end{align*}
We define the following sequence:
\begin{align*}
j(m_{2p-1},\ldots,m_1)&:=
(q^{1-N})_{m_{2p-1}}q^{-N m_{2p-1}}\\
& \ \ \times\prod_{i=1}^{2p-2}(-1)^{m_i} q^{(-1)^{i}N m_i+\binom{m_i}{2}
-m_im_{i+1}}{\begin{bmatrix}m_{i+1}\\m_i\end{bmatrix}}.
\end{align*}
Since for every $m_i, \, m_{i+1}\in\N$,
$$
{\begin{bmatrix}m_{i+1}\\m_i\end{bmatrix}}=0 \ \ (m_{i+1}<m_i)
$$
holds, we obtain the following equality: for every $1\leq i\leq 2p-2$,
$$
j(m_{2p-1},\ldots ,m_{i+1},m_i, \ldots,m_1)=0 \ \ (m_{i+1}<m_i).
$$
Hence,
\begin{align*}
&|\sum_{m_{2p-1},\ldots, m_1=0}^{N-1}j(m_{2p-1},\ldots ,m_1)|\\
&\geq 2\Bigl|| \sum_{m_{2p-2},\ldots, m_1=0}^{N-1}j(n_{2p-1},m_{2p-2}\ldots ,m_1)|
-|\sum_{m_{2p-2},\ldots, m_1=0}^{N-1} j(0,m_{2p-2},\ldots ,m_1)|\Bigr|\\
&\geq 2\Bigl(|\sum_{m_{2p-2},\ldots ,m_1=0}^{N-1}j(n_{2p-1},m_{2p-2},\ldots ,m_1)|
-\sum_{m_{2p-2},\ldots ,m_1=0}^{N-1}\Bigr)\\
&=2\Bigl(|\sum_{m_{2p-2},\ldots ,m_1=0}^{N-1}j(n_{2p-1},m_{2p-2},\ldots ,m_1)|
-N^{2p-2}\Bigr)\\
&\geq 2\Bigl(2\Bigl||\sum_{m_{2p-3},\ldots, m_1=0}^{N-1} j(n_{2p-1},n_{2p-2},m_{2p-3},\ldots,m_1)|\\
& \ \ -|\sum_{m_{2p-3},\ldots,m_1=0}^{N-1} 
j(n_{2p-1},n_{2p-1}+1,m_{2p-3},\ldots,m_1)|\Bigr|-N^{2p-2}\Bigr)\\
&=2\Bigl(2|\sum_{m_{2p-3},\ldots, m_1=0}^{N-1} j(n_{2p-1},n_{2p-2},m_{2p-3},\ldots,m_1)|-N^{2p-2}\Bigr)\\
&=4|\sum_{m_{2p-3},\ldots, m_1=0}^{N-1} j(n_{2p-1},n_{2p-2},m_{2p-3},\ldots,m_1)|-2N^{2p-2}.
\end{align*}
We repeat the above calculation $2p-3$ times. Then, we obtain the following inequality:
\begin{align*}
&\Bigl|\sum_{m_{2p-1},m_{2p-2},\ldots, m_1=0}^{N-1}
j(m_{2p-1}, m_{2p-2} ,\ldots,m_1)\Bigr|\\
&\geq 2^{2p-1}|j(n_{2p-1},n_{2p-2},n_{2p-3},\ldots, n_1)|-2N^{2p-2}.
\end{align*}
Hence,
$$
|J_N(K^{*}_{p>0};q)|\geq 2^{2p-1}|(q)_{0}|\prod_{i=1}^{2p-2} \Bigl|{\begin{bmatrix}n_{i+1} \\n_{i} \end{bmatrix}}\Bigr|-2N^{2p-2}.
$$
\end{proof}
\begin{lem}\label{115}
With the same notation as in Theorem \ref{155} we have that
$$
{\rm Im}[f(a_1,\ldots,a_{2p-1})]\leq \lim_{N\to \infty}\frac{2 \pi \log |J_{N}(K^{*}_{p>0};q)|}{N}.
$$
\end{lem}
\begin{proof}Because of Lemma \ref{121}, for a sufficiently large $N$, 
we obtain the following formula:
$$
2^{2p-1}|(q)_{0}|\prod_{i=1}^{2p-2} \Bigl|{\begin{bmatrix}n_{i+1} \\n_{i} \end{bmatrix}}\Bigr|
\leq |J_N(K^{*}_{p>0};q)|+2N^{2p-2}.
$$
Hence,
\begin{align} 
&\frac{2 \pi \log (|J_N(K^{*}_{p>0};q)|+2N^{2p-2})}{N} \geq \frac{2 \pi}{N}\log \Bigl(
2^{2p-1}|(q)_{0}|\prod_{i=1}^{2p-2} \Bigl|{\begin{bmatrix}n_{i+1} \\n_{i} \end{bmatrix}}\Bigr|\Bigr)
\label{116} \\ 
&= \frac{2 \pi}{N} \Bigl(\log 2^{2p-1}+\log |(q)_{0}|
+\sum_{i=1}^{2p-2}\log \Bigl|{\begin{bmatrix}n_{i+1}\\n_i\end{bmatrix}}\Bigr| \Bigr) \nonumber \\ \nonumber
&\geq \frac{2 \pi}{N} \Bigl(\log 2^{2p-1}+\log |(q)_{0}|+\log |(q)_{2p-1}|
+\sum_{i=1}^{2p-2}\log \Bigl|{\begin{bmatrix}n_{i+1}\\n_i\end{bmatrix}}\Bigr| \Bigr) \nonumber \\ \nonumber
&\overset{N}{\sim} \sum_{i=1}^{2p-2}{\rm Im}[{\rm Li}_2(\exp (2\pi\sqrt{-1}n_i/N))
-{\rm Li}_2(\exp (2\pi\sqrt{-1}n_{i+1}/N))\\ \nonumber
& \ \ +{\rm Li}_2(\exp (2\pi\sqrt{-1}(n_{i+1}-n_i)/N))]
+{\rm Im}[{\rm Li}_2(\exp (2\pi\sqrt{-1}n_{2p-1}/N))]\\
&= {\rm Im}[f(z_1^{(1)},\ldots,z_{2p-1}^{(2p-1)})]  \nonumber \\
&\overset{N}{\sim}
{\rm Im}[f({\rm exp}(2 \pi \sqrt{-1} n_1/N),\ldots ,{\rm exp}(2 \pi \sqrt{-1} n_{2p-1}/N))]. \label{117}
\end{align}
Here, Inequality (\ref{116}) is shown by using Lemma \ref{121}.
Equivalent relation (\ref{117}) is shown by using Lemma \ref{103}.
Hence,
\begin{align}\nonumber
&|J_N(K^{*}_{p>0};q)|+2N^{2p-2}\\ \label{124}
&\geq \exp (\frac{N}{2\pi}{\rm Im}[f(\exp (2\pi \sqrt{-1}n_1/N),\ldots ,
\exp (2\pi \sqrt{-1} n_{2p-1}/N))]).
\end{align}
Hence,
\begin{align*}
&\int_{0}^N\cdots \int_{0}^{N}
|(|J_N(K^{*}_{p>0};q)|+2N^{2p-2})| dn_1\cdots dn_{2p-1}\\
&\geq \int_0^N\cdots \int_0^N 
|\exp \Bigl(\frac{N}{2\pi}{\rm Im}[f(\exp (2\pi\sqrt{-1}n_1/N),\ldots ,\exp (2\pi \sqrt{-1}n_{2p-1}/N))]
\Bigr)|\\
& \ \ \times |\exp \Bigl(\frac{-\sqrt{-1}N}{2\pi}\\
& \ \ \times{\rm Re}[f(\exp (2\pi\sqrt{-1}n_1/N),\ldots ,\exp (2\pi \sqrt{-1}n_{2p-1}/N))]\Bigr)|
dn_1 \cdots dn_{2p-1}\\
&\geq \int_{0}^{N}\cdots \int_{0}^{N}
|\exp \Bigl(\frac{-\sqrt{-1}N}{2\pi}\\
& \ \ \times f(\exp (2\pi \sqrt{-1}n_1/N),\ldots ,
\exp (2\pi \sqrt{-1} n_{2p-1}/N))\Bigr)|dn_1\cdots dn_{2p-1}\\
&\geq \Bigl|\int_{0}^{N}\cdots \int_{0}^{N}
\exp \Bigl(\frac{-\sqrt{-1}N}{2\pi}\\
& \ \ \times f(\exp (2\pi \sqrt{-1}n_1/N),\ldots ,
\exp (2\pi \sqrt{-1} n_{2p-1}/N))\Bigr) dn_1\cdots dn_{2p-1}\Bigr|.
\end{align*}
Because of Corollary \ref{123}, we obtain the following formulas:
\begin{align*}
&N^{2p-1}
%\Bigl|\Bigl(\frac{N}{2\pi\sqrt{-1}}\Bigr)^{2p-1}\prod_{i=1}^{2p-1}\frac{1}{n_i}|\log \frac{\mu_i}{\lambda_i}|
(|J_N(K^{*}_{p>0};q)|+2N^{2p-2})\\
&\geq \Bigl| \int_{0}^N \cdots\int_{0}^N
%^{\frac{N\log \mu_{1}}{2\pi\sqrt{-1} n_1}}
\exp (\frac{-\sqrt{-1}N}{2\pi}\\
& \ \ \times f(\exp(2\pi \sqrt{-1}n_1/N),\ldots ,\exp(2\pi \sqrt{-1}n_{2p-1}/N)))
dn_1 \cdots dn_{2p-1}\Bigr|\\
&= \Bigl|\exp \Bigl(-\frac{\sqrt{-1}\pi (2p-1)}{4}\Bigr) \cdot N^{\frac{2p-1}{2}}
\cdot \exp \Bigl(\frac{N}{2\pi}{\rm Im}[f(a_1,\ldots ,a_{2p-1})]\Bigr)\\
& \ \ \times \Bigl(\frac{1}{a_1}\cdots\frac{1}{a_{2p-1}}+O(N^{-1}\Bigr)\\
& \ \ \times \prod_{j=1}^{2p-1}(-\xi_j)^{-\frac{1}{2}}
\exp \Bigl(\frac{-\sqrt{-1}N}{2\pi}{\rm Re}[f(a_1,\ldots,a_{2p-1})]\Bigr)\Bigr|\\
&=N^{\frac{2p-1}{2}}
\exp \Bigl(\frac{N}{2\pi}{\rm Im}[f(a_1,\ldots ,a_{2p-1})]\Bigr)\cdot U(N).
\end{align*}
Here,
$$
U(N):=\Bigl| \Bigl(\frac{1}{a_1}\cdots\frac{1}{a_{2p-1}}+O(N^{-1})\Bigr)
\prod_{j=1}^{2p-1}(-\xi_j)^{-\frac{1}{2}}\Bigr|.
$$
Hence,
$$
\frac{2\pi}{N}\Bigl(\log \frac{1}{U(N)}+\log N^{\frac{2p-2}{2}}
%+\sum_{i=1}^{2p-1}\log\frac{1}{n_i}|\log \frac{\mu_i}{\lambda_i}|
+\log \Bigl||J_N(K^{*}_{p>0};q)|+2N^{2p-2}\Bigr|\Bigr)
\geq {\rm Im}[f(a_1,\ldots ,a_{2p-1})].
$$
Because of Inequality (\ref{124}), we have that
\begin{align*}
\frac{2N^{2p-2}}{|J_N(K^{*}_{p>0};q)|}
&\leq \frac{2N^{2p-2}}{\exp (\frac{N}{2\pi}{\rm Im}[f(\exp (2\pi \sqrt{-1}n_1/N),\ldots ,
\exp (2\pi \sqrt{-1} n_{2p-1}/N))])}\\
&\to 0 \ \ (N\to \infty).
\end{align*}
Since
$$
\frac{2 \pi \log \Bigl||J_N(K^{*}_{p>0};q)|+2N^{2p-2}\Bigr|}{N}\overset{N}{\sim}
\frac{2 \pi \log |J_N(K^{*}_{p>0};q)|}{N},
$$
and
$$
\lim_{N\to \infty}\frac{2\pi}{N}\log N = \lim_{N\to \infty}\frac{2\pi}{N}\log \frac{1}{U(N)}=0
$$
hold, we obtain the following formula:
$$
{\rm Im}[f(a_1,\ldots,a_{2p-1})]\leq \lim_{N\to \infty}\frac{2 \pi \log |J_{N}(K^{*}_{p>0};q)|}{N}.
$$
\end{proof}

\begin{thm}\label{119}
Theorem \ref{97} is true.
\end{thm}
\begin{proof}
Because of Lemma \ref{99}, Lemma \ref{109}, and Lemma \ref{115}, we obtain Theorem \ref{97}.
\end{proof}

%%%%%%%%%%%%%%%%%%%%   End of main body of article
%
%                             References
%
%   BiBTeX users uncomment the following line:
%
%\bibliographystyle{gtart}
%

\end{document}